\documentclass[12pt, reqno, a4paper]{amsart}

\usepackage{ amssymb, amsmath, enumerate, amsfonts, amsthm, mathrsfs, url, bm, mathtools}

\numberwithin{equation}{section}

\usepackage{xcolor}  	
\usepackage[backref]{hyperref}
\hypersetup{
	colorlinks,
    linkcolor={blue!60!black},
    citecolor={blue!60!black},
    urlcolor={red!60!black}
}

\usepackage{color}

\usepackage[margin=1.25in]{geometry}

\RequirePackage{doi}

\usepackage[square,sort,comma,numbers]{natbib}
\makeatletter
\@namedef{subjclassname@2020}{%
  \textup{2020} Mathematics Subject Classification}
\makeatother

 \newcommand{\set}[1]{\left\{#1\right\}}
\newcommand{\bigset}[1]{\bigl\{ #1 \bigr\}}
\newcommand{\Bigset}[1]{\Bigl\{ #1 \Bigr\}}
\newcommand{\biggset}[1]{\biggl\{ #1 \biggr\}}

\newcommand{\abs}[1]{\left| #1\right|}
\newcommand{\bigabs}[1]{\bigl| #1 \bigr|}
\newcommand{\Bigabs}[1]{\Bigl| #1 \Bigr|}
\newcommand{\biggabs}[1]{\biggl| #1 \biggr|}

\newcommand{\floor}[1]{\left\lfloor #1 \right\rfloor}
\newcommand{\brac}[1]{\left( #1 \right)}
\newcommand{\bigbrac}[1]{\bigl( #1 \bigr)}
\newcommand{\Bigbrac}[1]{\Bigl( #1 \Bigr)}
\newcommand{\biggbrac}[1]{\biggl( #1 \biggr)}

\newcommand{\norm}[1]{\left\| #1\right\|}

\newcommand{\rd}{\,\mathrm{d}}

\newcommand{\twosum}[2]{ \sum_{\substack{#1\\ #2}}}
\newcommand{\threesum}[3]{ \sum_{\substack{#1\\ #2\\ #3}}}

\newcommand{\N}{\mathbb{N}}
\newcommand{\Z}{\mathbb{Z}}
\newcommand{\Q}{\mathbb{Q}}
\newcommand{\R}{\mathbb{R}}
\newcommand{\C}{\mathbb{C}}
\newcommand{\T}{\mathbb{T}}

\newcommand{\E}{\mathbb{E}}

\newcommand{\eps}{\varepsilon}

\makeatletter
\let\@@pmod\pmod
\DeclareRobustCommand{\pmod}{\@ifstar\@pmods\@@pmod}
\def\@pmods#1{\mkern4mu({\operator@font mod}\mkern 6mu#1)}
\makeatother

\newcommand{\poly}{\mathrm{poly}}
\newcommand{\lip}{\mathrm{Lip}}

\newcommand{\mrd}[1]{\,(\mathrm{mod}\,#1)}
\renewcommand{\bar}{\overline}

\renewcommand{\leq}{\leqslant}
\renewcommand{\geq}{\geqslant}
\renewcommand{\epsilon}{\varepsilon}

\newtheorem{theorem}{Theorem}[section]

\newtheorem{proposition}[theorem]{Proposition}
\newtheorem{lemma}[theorem]{Lemma}

\theoremstyle{definition}
\newtheorem{definition}[theorem]{Definition}
\newtheorem*{remark}{Remark}

\numberwithin{theorem}{section}

\begin{document}

\title[Discorrelation of multiplicative functions with nilsequences]{Discorrelation of multiplicative functions with nilsequences and its application on  coefficients of automorphic $L$-functions }

\author{Xiaoguang He}

\address{College of Mathematical Sciences, Sichuan University, Chengdu, Sichuan 610016,
PR China}

\email{hexiaoguangsdu@gmail.com}

\author{Mengdi Wang}

\address{Department of Mathematics, KTH, Stockholm 100 40, Sweden}

\email{mengdiw@kth.se}


\maketitle

\begin{abstract}
We introduce a class of multiplicative functions in which each function satisfies some statistic conditions, and then prove that the above functions are not correlated with finite degree polynomial nilsequences. Besides, we give two applications of this result. One is that the twisting of coefficients of automorphic $L$-function on $GL_m (m \geq 2)$ and polynomial nilsequences has logarithmic decay; the other is that the mean value of the M\"obius function, coefficients of automorphic $L$-function and polynomial nilsequences also has logarithmic decay.

\end{abstract}

\section{Introduction}

Let $a(n),b(n):\N \to\C$ be two complex valued sequences. Following the literature \cite{GT12b}, we say they  are ``strongly orthogonal" if
\[\sum_{n\leq N}a(n)b(n)=O_A\Bigbrac{(\log N)^{-A}\sum_{n\leq N}|a(n)b(n)|},
\]
holds for any $A>0$   and  uniformly for $N\geq 2$.  The study of the correlation  between two sequences is an important subject in number theory. Here, we list some well-known pieces of literature, taking one of the sequences being the M\"obius function as an example.
\begin{enumerate}[$\bullet$]
\item For $(a(n),b(n))=(\mu(n),1)$, the strong orthogonality $\sum_{n\leq N}\mu(n)\ll Ne^{-c\sqrt{ \log N}}$ is essentially equivalent to the prime number theorem;
	\item For $(a(n),b(n))=(\mu(n),e(n\alpha))$, this is a classical result due to Davenport \cite{Dav}, he proved a strongly orthogonal result by modifying Vinogradov's method on bilinear forms;
	\item For $(a(n),b(n))=(\mu(n),e(n^k\alpha))$, Hua \cite{Hua} generalizes Davenport's strong orthogonality result from linear phase functions to polynomial phase functions;
	\item For $(a(n),b(n))=(\mu(n),F(g(n)\Gamma))$, where $F(g(n)\Gamma)$ is a nilsequence, Green and Tao \cite{GT12b} prove that the M\"obius function is strongly orthogonal to polynomial nilsequences.
\end{enumerate}

 In this paper, we consider the correlation of nilsequences with a class of multiplicative functions. And in the rest of this section, we would pay attention to the historical developments that are more relevant to the theme of this paper.

Suppose that $f:\N\to\C$ is a 1-bounded multiplicative complex valued function, Daboussi \cite{Dab} proves that for any irrational frequency $\alpha\in\R/\Z(=\T)$,
\[
S(f,\alpha):=\frac{1}{N}\sum_{n\leq N}f(n)e(n\alpha)=o(1).
\]
 Montgomery and Vaughan \cite{MV} consider a much more general class of multiplicative functions. They relax the $1$-bounded condition to the below two conditions
\begin{equation}\label{bound-prime}
	|f(p)|\leq A, \textrm{ for all primes} \ p,
\end{equation}
and
\begin{equation}\label{bound-l2}
	\sum_{n\leq N}|f(n)|^2\leq A^2N,  \text{ for large  natural numbers} \ N,
\end{equation}
where $A\geq 1$ is some constant. More formally, they prove that if $\norm{q\alpha}\leq\frac{1}{q}$ for some integer $1\leq q\leq N(\log N)^{-3}$,  then for every function $f$ satisfies conditions (\ref{bound-prime}) and (\ref{bound-l2}), we have
\[
S(f,\alpha)\ll \frac{1}{\sqrt{\phi(q)}}+\frac{1}{\log N}.
\]
  Taking note that when $q\geq(\log N)^{2+\eps}$, the above upper bound would be determined by the second term. Thus, when the function $f$ retains reasonable decay in arithmetic progressions with modulus less than $(\log N)^{2+\eps}$, we can expect that for any frequency $\alpha\in\T$,
\begin{align}\label{aim}
\frac{1}{N}\sum_{n\leq N}\bigbrac{f(n)-\E_{[N]}f}e(n\alpha)\ll (\log N)^{-1},
\end{align}
 where we have written $\E_{[N]}f=\E_{n\in[N]}f(n)=\frac{1}{N}\sum_{n\leq N}f(n)$ as the average of $f$ on the discrete interval $[N]=\set{1,2,\dots,N}$. Recently, Jiang, L\"{u} and Wang \cite{JLW} weaken the above condition (\ref{bound-prime}) that $f$ takes bounded values in primes to the following two conditions
\begin{equation}\label{p-l2}
	\sum_{p\leq N}|f(p)|^2\log p\ll N;
\end{equation}
and
\begin{equation}\label{ph-l2}
	\twosum{p\leq N}{p+h\text{ is  prime}}|f(p)f(p+h)|\ll\frac{h}{\phi(h)}\frac{N}{(\log N)^2}
\end{equation}
holds for all positive integers $h$. They can show that when $f$ is under conditions (\ref{bound-l2}), (\ref{p-l2}) and (\ref{ph-l2}), the error term of $S(f,\alpha)$ is similar to the result of Montgomery and Vaughan. Using this result, they also prove that the mean value of coefficients of automorphic $L$-functions and linear phase functions has logarithmic decay (see next subsection).

The studying of the correlation between multiplicative functions and polynomial nilsequence in place of the exponential function $n\to e(n\alpha)$ begins with Green and Tao \cite{GT10}. In that paper, they obtained the asymptotic formula of linear equations in primes using the information that $W$-tricked von Mangoldt function does not correlate with polynomial nilsequences. And very recently, Tao and Ter\"av\"ainen \cite{TT} give the above asymptotic formula a quantitative error term.

 In order to make this paper clear and readable, we are planning to record the definition of nilsequences and related notations here, and refer to \cite{GT12a} for more detailed introductions.

 \begin{definition}
	Let $G$ be a connected, simply-connected nilpotent Lie group, and let $\Gamma\leq G$ be a lattice. By a \emph{ filtration} $G_{\bullet}=(G_i)_{i=0}^{\infty}$ on $G$, we mean a descending sequence of groups $G=G_1\supseteq G_2\supseteq\cdots\supseteq G_d\supseteq G_{d+1}=\{\mathrm{id}_G\}$ such that \begin{equation}\label{EqFiltration}[G,G_{i-1}]\subseteq G_i, \forall i\geq 2.\end{equation} This actually implies $[G_i,G_j]\subseteq G_{i+j}$ for all $i, j\geq 1$, and  $\Gamma_i:=\Gamma\cap G_i$ is a lattice in $G_i$ for $i\geq 0$. The number $d$ is the \emph{ degree} of the filtration $G_\bullet$.  The \emph{ step} $s$ of $G$ is the degree of the lower central filtration defined by $G_{i+1}=[G,G_i]$. A lattice $\Gamma$ must be cocompact, and the compact quotient $G/\Gamma$ is called a \emph{ nilmanifold}.

	We say that $g$ is a \emph{ polynomial sequence} with coefficients in $G_{\bullet}$, and write $g \in \poly(\Z,G_{\bullet})$, if $g:\mathbb{Z}\to G$ satisfies the derivative condition
	\begin{align*}
		\partial_{h_1}\cdots \partial_{h_i}g(n) \in G_i
	\end{align*}
	for all $i\geq 0$, $n\in \Z$ and all $h_1,\ldots, h_i\in \mathbb{Z}$, where $\partial_h g(n):=g(n+h)g(n)^{-1}$ is the discrete derivative with shift $h$.
	
	The Mal'cev basis $\mathcal{X}$ (see next section) induces a right invariant metric $d_G$ on $G$, which is the largest metric such that $d(x,y)\leq|\psi_\mathcal{X}(xy^{-1})|$ always holds, where $|\cdot|$ denotes the $l^\infty$-norm on $\mathbb{R}^m$, and $\psi_\mathcal{X}:G\to\R$ is the Mal'cev coordinate map which is defined in \cite[(2.1)]{GT12a}. Actually, this in turn induces a metric $d_{G/\Gamma}$ on $G/\Gamma$. For a function $F:G/\Gamma\to\mathbb{C}$, we define its  \emph{ Lipschitz norm} as
	 \begin{equation}\label{DefLip}\|F\|_{\operatorname{Lip}}=\|F\|_{\infty}+\sup_{x,y\in G/\Gamma, x\neq y}\frac{|F(x)-F(y)|}{d_{G/\Gamma}(x,y)}\end{equation}
	  with respect to $d_{G/\Gamma}$.
	
	Finally, if $F:G/\Gamma\to \mathbb{C}$ is a Lipschitz function (that is $\norm{F}_{\lip}<\infty$), we call a sequence of the form $n \mapsto F(g(n)\Gamma)$ a \emph{nilsequence}.
\end{definition}

An interesting question is to generalize the above-described classical results on linear phase functions from Daboussi, Montgomery and Vaughan and so on, to nilsequences. Frantzikinakis and Host \cite[Theorem 2.2]{FH} generalized Daboussi's consideration to nilsequence cases, and they can prove a qualitative result as follows.

\begin{theorem}[Daboussi for nilsequences]\label{FH}
Let $G/\Gamma$ be a nilmainifold with finite dimension and $G_\bullet$ be a finite degree filtration of $G$. Suppose that $g\in\poly(\Z,G_\bullet)$ is a polynomial sequence and $\bigbrac{g(n)\Gamma}_{n\in\N}$ is totally 	equidistributed\footnote{An infinite sequence $(g(n)\Gamma)_{n\in\N}$ in $G/\Gamma$ is said to be totally equidistributed if for any integers $0\neq a\in\Z$ and $r\in\Z$ and any continuous function $F:G/\Gamma\to\C$, we have
\[
\lim_{n\to\infty}\E_{n\in[N]}F(g(an+r)\Gamma)=\int_{G/\Gamma}F.
\]
Here, the symbol $\int_{G/\Gamma}$ will stand for integration with respect to the unique left invariant probability measure on $G/\Gamma$.}
 in $G/\Gamma$, then for any 1-bounded multiplicative function $f:\Z\to\C$ and any continuous function $F:G/\Gamma\to\C$ with $\int_{G/\Gamma}F=0$, we have
\[
\lim_{N\to\infty}\E_{n\in[N]}f(n)F(g(n)\Gamma)=0.
\]
\end{theorem}

Matthiesen \cite{Mat} consideres a general class of functions which doesn't need the functions to take bounded values at every integer point but bound at primes and higher power of primes. Let $H\geq1$ be a fixed number, $\mathcal M_H$ denote the class of multiplicative functions in which each function $f:\N\to\C$ satisfies the following conditions:
\begin{enumerate}
\item (Bound at prime powers). $|f(p^k)|\leq H^k$ for all prime powers $p^k$;
\item  (Accumulation at primes is positive). There is a number $0<\alpha_f\leq1$ such that the below inequality holds for all large $N$
\[
\frac{1}{N}\sum_{p\leq N}|f(p)|\log p\geq\alpha_f.
\]	
\end{enumerate}
Matthiesen \cite{Mat} then gives a quantitative discorrelation of polynomial nilsequences with $W$-tricked functions from the class $\mathcal M_H$, see \cite[Theorem 6.1]{Mat} for precise statement.

We plan to study another class of multiplicative functions, for which we relax the point-to-point information at  prime powers but instead we need stronger statistics information.

Let $\mathcal M'$ be the class of multiplicative functions $f:\N\to\C$ with the following properties,
\begin{enumerate}
\item (Equidistributed in arithmetic progressions with small moduli). There are relatively prime integers $1\leq b\leq W\ll(\log N)^C$  such that the $W$-tricked function of $f$, named $f(W\cdot+b)$, is equidistributed in long arithmetic progressions. More formally,
\begin{align}\label{w-equi}
\frac{\phi(W)}{W}\biggabs{\E_{n\in P}\biggbrac{f(Wn+b)-\frac{\phi(W)}{WN}\sum_{n\in[N]}f(Wn+b)}}\ll\frac{1}{\log N},
\end{align}
where $P\subseteq [N]$ is an arithmetic progression of length $|P|\gg N/(\log N)^C$.
\item (The $L^2$-norm at primes is bounded).
\begin{align}\label{lp2}
\sum_{p\leq N}|f(p)|^2\log p\ll N.
 \end{align}	
\item (The $L^2$-norm is bounded).
\begin{align}\label{fl2}
\E_{n\in[N]}|f(n)|^2\ll1;
\end{align}
and
\begin{align}\label{wl2}
\frac{\phi(W)}{WN}\sum_{n\leq N}|f(Wn+b')|^2\ll1
\end{align}
for all integers $1\leq b'\leq W$ coprime with $W$.
\end{enumerate}

 \cite[Proposition A.2]{GT12b} shows that the M\"obius function is equidistributed in  progressions with small common differences. It can also been seen from \cite[Proposition 2.2]{TT} that $\Lambda-\Lambda_{\mathrm{Siegel}}$ is equidistributed in progressions with small moduli, where $\Lambda$ is the von Mangoldt function and $\Lambda_{\mathrm{Siegel}}$ is defined in \cite[Definition 2.1]{TT}. When taking $W=1$ in (\ref{w-equi}), the von Mangoldt function itself is not equidistributed in progressions with small moduli, but this difficulty can be overcome by a simple affine change of variables known as $W$-trick, which is introduced in \cite{GT08}. However, for general multiplicative functions, finding such a number $W$ such that $f(W\cdot+b)$ eliminates bias to any residue class with small modulus is  much more complicated, we refer the interested readers to \cite[Section 5]{Mat} for the inspiring readings.

\begin{remark}
When $W\neq1$ and $W\ll(\log N)^C$, for the function $f(W\cdot+b)$, the logarithmic decay in condition (\ref{w-equi}) is, somehow,  also a strict condition, we can weaken this condition to expect a poor error term of Theorem \ref{main}. However, we are extremely interested in how to obtain a result similar to Montgomery and Vaughan (looks like (\ref{aim})) in the nilsequence settings. And this is the reason that we do such a logarithmic decay hypothesis.	
\end{remark}

\begin{remark}
Assumption (\ref{wl2}) may be strange at first glance, we try to defend a bit for it. The application of the approach of Montgomery and Vaughan would reduce the question to considering a sum over $pn$ in an arithmetic progression, where $p$ is a prime. Visually speaking, to handle $\sum_{pn\leq x; pn\equiv b\mrd W}f(p)f(n)e(pn\alpha)$, where $f(n)$ and $f(p)$ may be seen as additive coefficients for the moment. When we aim to remove these coefficients and so to separate variables $p$ and $n$, taking note that the summation is on a reduced residue class,  at least one of the variables $p$ and $n$ needs to run over all of the reduced residue classes modulus $W$, and this leads to  condition  (\ref{wl2}) after applying Cauchy-Schwarz inequality. Moreover, it is easy to verify that when $W=1$ the condition (\ref{wl2}) is indeed (\ref{fl2}).
\end{remark}

\begin{theorem}\label{main}
Let $G/\Gamma$ be a nilmanifold of dimension $m_G\geq1$, $\mathcal X$ be a $M_0$-rational Mal'cev basis adapted to $G/\Gamma$ for some $2\leq M_0\leq \log N$, and $G_\bullet$ be a filtration of $G$ of degree $d\geq1$. Suppose that $g\in\poly(\Z,G_\bullet)$ is a polynomial sequence and $F:G/\Gamma\to\C$ is a 1-bounded Lipschitz function, then for every function $f\in\mathcal M'$, we have
\[
\frac{\phi(W)}{WN}\sum_{n\in[N]}\Bigbrac{f(Wn+b)-\E_f(N;W,b)}F(g(n)\Gamma)\ll_{m_G,d}(1+\norm{F}_{\lip})\frac{1}{\log N},
\]
where we have written $\E_f(N;W,b)=\frac{\phi(W)}{WN}\sum_{n\in[N]}f(Wn+b)$.\end{theorem}

\subsection*{Coefficients of automorphic $L$-functions twisted by nilsequences}

Inspired from the argument of \cite{JLW}, we aim to apply Theorem \ref{main} to the coefficients of automorphic $L$-functions. We begin with a brief introduction to the coefficients of automorphic $L$-functions.

 Let $m\geq 2$ be an integer and $\pi$ be an automorphic irreducible cuspidal representation of $GL_m$ over $\Q$ with unitary central character. Denote by $\lambda_\pi (n)$ the Dirichlet coefficients of automorphic $L$-function $L(s,\pi)$ attached to $\pi$. Then $\lambda_\pi(n)$ is a multiplicative function. We now write $S_1(\alpha,N)$ as the sum of $\lambda_\pi$ with the linear phase function, i.e.
 \[
 S_1(\alpha,N)=\sum_{n\leq N}\lambda_\pi(n)e(n\alpha).
 \]
 The Ramanujan conjecture asserts that
\[|\lambda_\pi(p)|\leq m\]
for all primes $p$, and this means that $\lambda_\pi(n)$ satisfy the condition (\ref{bound-prime}) under Ramanujan conjecture. Besides, in the paper of Jiang, L\"u and Wang \cite{JLW}, they proved that $\lambda_\pi(n)$ also satisfy condition (\ref{bound-l2}), thus, under Ramanujan conjecture, the result given by Montgomery and Vaughan can be applied to $\lambda_\pi(n)$. Whilst, we also concerne about the upper bound of $S_1(\alpha,N)$ without any sharp hypothesis.

 When $m=2$, $\lambda_\pi(n)$ are the normalised coefficients of a modular form or Maass form, and in this case, \cite[Theorem 8.1]{Iwa} shows that
\[S_1(\alpha,N)\ll N^{\frac{1}{2}}(\log N).\]
We now view $S_1(\alpha,N)$ as the Fourier coefficient of $\lambda_\pi$ at frequency $\alpha$ temporarily, then using of Parseval's identity and the Rankin-Selberg theory yields that
\[\int_0^1|S_1(\alpha,N)|^2d\alpha=\sum_{n\leq N}|\lambda_\pi(n)|^2\sim N,\]
which means that, on average, $S_1(\alpha,N)$ demonstrates square-root cancellation. As a consequence, one cannot expect above exponent $\frac{1}{2}$ could be reduced for all frequencies $\alpha$.

When $m=3$, Miller \cite{Mil} proved that \[S_1(\alpha,N)\ll N^{\frac{3}{4}+\varepsilon},\]
by using of the Vorono\"i summation for $GL_3$. For general $m\geq 4$, Jiang, L\"u and Wang \cite[Theorem 1.2]{JLW} proved that
\[S_1(\alpha,N)\ll \frac{N}{\log N}.\]
Let's give a quick sketch of their proof. In order to get the desired bound,  they firstly split frequencies $\alpha\in\T$ into the so-called $major \ arcs$ and $minor \ arcs$. When $\alpha$ belongs to $major \ arcs$, they can get good bounds for $S_1(\alpha, N)$ much easier since the good behavior of $\lambda_\pi$ in progressions with small modulus. As for those frequencies in $minor \ arcs$, they employ their Theorem 1.1 in  \cite{JLW} to the coefficients $\lambda_\pi(n)$, thus, reduce the question  to verifying conditions (\ref{bound-l2}), (\ref{p-l2}) and (\ref{ph-l2}) for coefficients $\lambda_\pi(n)$. And, the main business of their paper is to handle the much harder condition (\ref{ph-l2}).  It is worth pointing out that the logarithmic decay for $S_1(\alpha,N)$ in their result is derived from the approach of Montgomery and Vaughan. And we believe it is a difficult work to gain any power saving result for $S_1(\alpha,N)$, and also needs more clever ideas.

There are also several pieces of literature concerning the correlation of coefficients $\lambda_\pi(n)$ and polynomial phase functions even nilsequences, such as the sum
\[
S_k(\alpha,N)=\sum_{n\leq N}\lambda_\pi(n)e(n^k\alpha).
\]
Cafferata, Perelli and Zaccagnini \cite{CPZ} extend the criterion of Bourgain--Sarnak--Ziegler by relaxing some conditions in the criterion to make it efficient for more multiplicative functions, then they applied this (generalized) criterion to holomorphic cusp form $\pi$ on $GL_2$ and prove that
\[S_k(\alpha,N)\ll N\frac{\log\log N}{\log N}.\]
Later, Jiang and L\"u \cite{JL} extend the above result to $GL_m$ automorphic cuspidal representation $\pi$ under Hypothesis H (Hypothesis H was stated by Rudnick and Sarnak \cite{RS}).   Afterward, the Hypothesis H is removed by Jiang and L\"u in a later paper \cite[Lemma 5.3]{JL21}.

As for general nilsequence settings, the only literature we can find is the paper of  Matthiesen \cite{Mat}. Thanks to her Theorem 6.1, one can show that, for each holomorphic cusp form $\pi$ on $GL_2$, the mean value of $\lambda_\pi(n)$ and polynomial nilsequences has some quantitative  decay.

Our second result of this paper is to extend the previous consideration to the correlation of general  $GL_m (m\geq 2)$ automorphic cuspidal representation $\pi$ over $\Q$ with polynomial nilsequence.

\begin{theorem}\label{lfunction}
Let $G/\Gamma$ be a nilmanifold of dimension $m_G\geq1$, $\mathcal X$ be a $M_0$-rational Mal'cev basis adapted to $G/\Gamma$ for some $2\leq M_0\leq \log N$, and $G_\bullet$ be a filtration of $G$ of degree $d\geq1$. Suppose that $g\in\poly(\Z,G_\bullet)$ is a polynomial sequence and $F:G/\Gamma\to\C$ is a 1-bounded Lipschitz function, then  we have
	\[
	\frac{1}{N}\sum_{n\in[N]}\lambda_\pi(n)F(g(n)\Gamma)\ll_{m_G,d,\pi}\bigbrac{1+\norm{F}_{\lip}}\frac{1}{\log N}.
	\]
\end{theorem}

\begin{remark}
To prove Theorem \ref{lfunction}, we need to verify $\lambda_\pi(n)$ satisfy the major-arc condition (\ref{w-equi}) with $W=1$ and the $L^2$-norm bounded conditions (\ref{lp2}) and (\ref{fl2}). Thus, compared with the proof of \cite{JLW}, we lessen the work to  verify $\lambda_\pi(n)$ satisfying the sieve condition (\ref{ph-l2}).
\end{remark}

We are willing to take some examples for the readers who are less familiar with nilsequences. Theorem 1.4 shows that the twisting of $\lambda_\pi(n)$ with a  polynomial phase function of degree $d\geq1$ has logarithmic decay, i.e.
\[
\sum_{n\leq N}\lambda_\pi(n)e(\alpha_d n^d+\dots+\alpha_1n+\alpha_0)\ll\frac{N}{\log N},
\]
where $\alpha_d,\dots,\alpha_1,\alpha_0\in\T$ are arbitrary frequencies; Theorem 1.4 also shows that the following expression  has logarithmic decay,
\[
\sum_{n\leq N}\lambda_\pi(n)e\bigbrac{\beta n\floor{n\alpha}}\psi(\{\alpha n\})\psi(\{\beta n\})\ll\frac{N}{\log N},
\]
where $\floor{\cdot}$ is the integer part operator, $\{\cdot\}$ is the fractional part operator, $\alpha,\beta\in\T$ are frequencies, and $\psi:[0,1]\to \C$ is a Lipschitz function that vanishes near 0 and 1.

It is also an interesting question to consider whether the sequences $\{\mu(n)\}_n$ and $\{\lambda_{\pi}(n)F(g(n)\Gamma)\}_n$ are orthogonal or not. This is motivated by the M\"obius pseudorandomness principle, which predicts that sums of the form $\sum_{n\leq x}\mu(n)\xi(n)$ should exhibit some cancellation if $\xi(n)$ is not obviously related to the prime factorisation of $n$. See also the M\"obius disjointness conjecture of Sarnak \cite{Sar}, which expects to use observables from zero entropy topological dynamical systems as the sequence $\xi(n)$. Our next conclusion shows that  the M\"obius pseudorandomness principle is true for the sequence $\{\lambda_{\pi}(n)F(g(n)\Gamma)\}$.

In this direction, the current recording results are from Fouvry and Ganguly \cite{FG} when $m=2$ and Jiang and L\"u \cite{JL19} when $m=3$ respectively. Combining their results, one has when $m=2,3$,
\[\sum_{n \leqslant N} \mu(n) \lambda_{\pi}(n) e(n \alpha) \ll N \exp (-c \sqrt{\log N}),\]
where $c>0$ is an absolute constant. In the same paper of Cafferata, Perelli and Zaccagnini \cite{CPZ}, they also have
\[\sum_{n \leqslant N} \mu(n) \lambda_{\pi}(n) e\left(n^{k} \alpha\right) \ll N \frac{\log \log N}{\sqrt{\log N}}.\]
 Jiang and L\"u \cite{JL} extend the above result to $GL_m$ automorphic cuspidal representation $\pi$ under Hypothesis H and Hypothesis S. Afterward, these two hypothesizes are removed in two later papers respectively--Jiang and L\"u \cite[Lemma 5.3]{JL21} removes Hypothesis H and Jiang, L\"u, Thorner and Wang \cite[Corollary 4.7]{JLTW} removes Hypothesis S.

Recently, Jiang, L\"u and Wang \cite{JLW21}\footnote{We would like to thank Yujiao Jiang for posting their paper \cite{JLW21} and explaining their papers.} prove that
\[\sum_{n \leqslant N} \mu(n) \lambda_{\pi}(n) e(n \alpha) \ll \frac{N}{\log N},\]
for $\pi$ is self-dual and $\pi \ncong \pi\otimes\chi$ for any quadratic primitive character $\chi$. In this direction, we can show that 
\begin{theorem}\label{mulfunction}
	Assume that $\pi$ is self-dual and $\pi \ncong \pi\otimes\chi$ for any quadratic primitive character $\chi$, with the notations in Theorem \ref{lfunction},  we have
	\[\sum_{n \leqslant N} \mu(n) \lambda_{\pi}(n) F(g(n)\Gamma) \ll_{m_G,d,\pi} \frac{N}{\log N}.\]
\end{theorem}

\begin{remark}
The assumption $\pi \ncong \pi\otimes\chi$ is in \cite[Theorem 1.2]{JLTW}.  Roughly speaking, Jiang, L\"u, Thorner and Wang  can  obtain an analogy of the Siegel theorem in higher rank groups based on this assumption.
\end{remark}

\begin{remark}
Following the argument in \cite{FG} and \cite{JL19}, we don't need the additional assumption on $\pi$ for $m=2$ and $m=3$ respectively.
\end{remark}

\subsection*{Acknowledgements}
The authors are grateful to Yujiao Jiang, Lilian Matthiesen, Joni Ter\"av\"ainen and Zhiwei Wang for their helpful discussions. The authors also appreciate comments and suggestions from the referee which help us to improve this article. Part of this work was written when M.W. visits Shandong University, she thanks the warm and generous hospitality from Shandong University. X.H. is supported by the National Postdoctoral Innovative Talents Support Program (Grants No. BX20190227), NSFC (No. 12101427) and the Fundamental Research Funds for the Central Universities, SCU (No. 2021SCU12109). M.W. is partially supported by the G\"oran Gustafsson Foundation.

\section{Preliminaries and Outline of the Proof}

\subsection{Preliminaries}

In this subsection, we quickly collect all the facts and notations that we will need from Green-Tao's paper \cite{GT12a}.

Denote by $\mathfrak{g}_i$ the Lie algebra $G_i$, then $\mathfrak{g}_\bullet=\{\mathfrak{g}_i\}$ is a \emph{ filtration} of Lie algebras, i.e. $[\mathfrak{g},\mathfrak{g}_i]\subseteq\mathfrak{g}_{i+1}$, if and only if $G_\bullet=\{G_i\}$ is a filtration.

\begin{definition}\cite[Definition 2.1]{GT12a}\label{malcev}
	Let $G/\Gamma$ be  an $m$-dimensional nilmanifold and let $G_{\bullet}$ be a filtration. A basis $\mathcal{X} = \{X_1,\dots,X_{m}\}$ for the Lie algebra $\mathfrak{g}$ over $\R$ is called a \emph{Mal'cev basis} for $G/\Gamma$ adapted to $G_{\bullet}$ if the following four conditions are satisfied:
	\begin{enumerate}[(i)]
		\item $\{X_j,X_{j+1},\cdots,X_m\}$ spans an ideal of $\mathfrak{g}$ for all $0\leq j\leq m$;
		\item For each $1\leq i\leq d$ and $m_i=\dim G_i$, the Lie algebra $\mathfrak{g}_i$ of $G_i$ is the linear span of  $\{X_{m-m_i+1},X_{m-m_i+2},\cdots,X_m\}$;
		\item There is a diffeomorphism $\psi_\mathcal{X}:G\to\mathbb{R}^m$ determined by $$\psi_\mathcal{X}\Big(\exp(\omega_1X_1)\cdots\exp(\omega_mX_m)\Big)=(\omega_1,\cdots,\omega_m);$$
		\item In the coordinate system $\psi_\mathcal{X}$, $\Gamma=\psi_\mathcal{X}^{-1}(\mathbb{Z}^m)$.
	\end{enumerate}
\end{definition}

We say that a Mal'cev basis $\mathcal{X}$ for $G/\Gamma$ is \emph{$Q$-rational} if all of the structure constants $c_{ijk}$ in the relations
\[ [X_i, X_j] = \sum_k c_{ijk} X_k\] are rational with height at most $Q$. The \emph{height} of a real number $x$ is defined as $\max(|a|,|b|)$ if $x=a/b$ is rational in reduced form.

\begin{definition}\cite[Definition 1.2]{GT12a}\label{almost-equidistribution}
  Let $G/\Gamma$ be a nilmanifold.
	\begin{enumerate}
		\item
		Given a length $N > 0$ and an error tolerance $\delta > 0$, a finite sequence $(g(n)\Gamma)_{n \in [N]}$ is said to be \emph{$\delta$-equidistributed} if we have
		$$ \left|\E_{n \in [N]} F(g(n) \Gamma) - \int_{G/\Gamma} F\right| \leq \delta \|F\|_{\operatorname{Lip}}$$
		for all Lipschitz functions $F: G/\Gamma \to \C$.
		\item A finite sequence $(g(n)\Gamma)_{n \in [N]}$ is said to be \emph{totally $\delta$-equidistributed} if we have
		$$ \left|\E_{n \in P} F(g(n) \Gamma) - \int_{G/\Gamma} F\right| \leq \delta \|F\|_{\operatorname{Lip}}$$
		for all Lipschitz functions $F: G/\Gamma \to \C$ and all arithmetic progressions $P \subset [N]$ of length at least $\delta N$.
	\end{enumerate}
\end{definition}

The symbol $\int_{G/\Gamma}$ will stand for integration with respect to the unique left invariant probability measure on $G/\Gamma$.

\begin{definition}\cite[Definition 1.17]{GT12a}\label{rat-def-quant}
	Let $G/\Gamma$ be a nilmanifold and let $Q > 0$ be a parameter. We say that $\gamma \in G$ is \emph{$Q$-rational} if $\gamma^r \in \Gamma$ for some integer $r$, $0 < r \leq Q$. A \emph{$Q$-rational point} is any point in $G/\Gamma$ of the form $\gamma\Gamma$ for some $Q$-rational group element $\gamma$. A sequence $(\gamma(n))_{n \in \Z}$ is \emph{$Q$-rational} if every element $\gamma(n)\Gamma$ in the sequence is a $Q$-rational point.
\end{definition}

\begin{definition}\cite[Definition 1.18]{GT12a}\label{smooth-seq-def}  Let $G/\Gamma$ be a nilmanifold with a Mal'cev basis $\mathcal{X}$.  Let $(\varepsilon(n))_{n \in \Z}$ be a sequence in $G$, and let $M, N \geq 1$.  We say that $(\varepsilon(n))_{n \in \Z}$ is \emph{$(M,N)$-smooth} if we have $d(\varepsilon(n),\operatorname{id}_G) \leq M$ and $d(\varepsilon(n),\varepsilon(n-1)) \leq M/N$ for all $n \in [N]$.
\end{definition}

With the above notations, we give the following factorization theorem. See \cite[Theorem 1.19]{GT12a}

\begin{proposition}[Green--Tao factorization theorem]\label{factorization}
	Let $m,d \geq 0$, and let $M_0, N \geq 1$ and $A > 0$ be real numbers.
	Suppose that $G/\Gamma$ is an $m$-dimensional nilmanifold together with a filtration $G_{\bullet}$ of degree $d$. Suppose that $\mathcal{X}$ is an $M_0$-rational Mal'cev basis $\mathcal X$ adapted to $G_{\bullet}$ and that $g \in \poly(\Z,G_{\bullet})$. Then there is an integer $M$ with $M_0 \leq M \ll M_0^{O_{A,m,d}(1)}$, a rational subgroup $G' \subseteq G$, a Mal'cev basis $\mathcal{X}'$ for $G'/\Gamma'$ in which each element is an $M$-rational combination of the elements of $\mathcal{X}$, and a decomposition $g = \varepsilon g' \gamma$ into polynomial sequences $\varepsilon, g', \gamma \in \poly(\Z,G_{\bullet})$ with the following properties:
	\begin{enumerate}
		\item $\varepsilon : \Z \rightarrow G$ is  $(M,N)$-smooth;
		\item $g' : \Z \rightarrow G'$ takes values in $G'$, and the finite sequence $(g'(n)\Gamma')_{n \in [N]}$ is totally $1/M^A$-equidistributed in $G'/\Gamma'$, using the metric $d_{\mathcal{X}'}$ on $G'/\Gamma'$;
		\item $\gamma: \Z \rightarrow G$ is $M$-rational, and $(\gamma(n)\Gamma)_{n \in \Z}$ is periodic with period at most $M$.
	\end{enumerate}
\end{proposition}

The next result is the key theorem to verify whether a sequence is equidistributed or not, but before it,   we introduce the so-called smoothness norm.

\begin{definition}\cite[Definition 2.7]{GT12a}\label{smooth-norm}
Suppose that $g:\Z\to\T$ is a polynomial sequence of degree $d$. Then $g$ may  be written uniquely as
\[
g(n)=\alpha_0+\alpha_1\binom n1+\dots+\alpha_d\binom nd.
\]
For any $N>0$, we define the \emph{smoothness norm} as
\[
\norm{g}_{C^\infty[N]}:=\sup_{1\leq j\leq d}N^j\norm{\alpha_j}.
\]	
\end{definition}

\begin{proposition}[Quantitative Leibman theorem]\label{leibman}
Let $m_G, d \geq 0$, $0 < \delta < 1/2$, and $N \geq 1$.
Let $G/\Gamma$ be an $m_G$-dimensional nilmanifold together with a filtration $G_{\bullet}$ of degree $d$ and a $\frac{1}{\delta}$-rational Mal'cev basis $\mathcal{X}$ adapted to this filtration. Suppose that $g \in \poly(\Z,G_{\bullet})$. If $(g(n)\Gamma)_{n \in [N]}$ is not $\delta$-equidistributed in $G/\Gamma$, then there is a horizontal character $\eta$ with $0<|\eta|\ll\delta^{-O_{m_G,d}(1)}$ such that
\[
\norm{\eta\circ g}_{C^\infty[N]}\ll\delta^{-O_{m_G,d}(1)}.
\]
\	
\end{proposition}
\begin{proof}
See	\cite[Theorem 2.9]{GT12a}.
\end{proof}

\subsection{Motivation and outline of the argument}

For convenience, in the rest of this paper we would allow all of the asymptotic notions, such as $O(\cdot),o(\cdot),\gg,\ll$ and so on, to depend on the degree $d$ and the dimension $m_G$.

We actually dawn on the ideas from the three papers \cite{GT12b}, \cite{Mat} and \cite{MV} to prove Theorem \ref{main}. As we said in introduction, Montgomery and Vaughan \cite{MV} deal with the correlation of a class of multiplicative functions with linear phase functions. For doing this, they first reduce the exponential sum $\E_{n\in[N]}f(n)e(n\alpha)$ with $\alpha\neq0$ to the one in which variables are in terms of primes $p$, that is $\E_{p\sim U}\E_{n\leq N/p}\log p\,f(p)f(n)e(pn\alpha)$. The using of Cauchy-Schwarz inequality shows that it is bounded by $\norm{f}_{L^2[\frac{N}{U}]}$ times a half power of
\[
\E_{p,p'\sim U}\log p\log p'\,f(p)f(p')\E_{n\leq\min\set{\frac{N}{p},\frac{N}{p'}}}e\bigbrac{(p-p')n\alpha}.
\]
Because the number of pairs $(p,p')\in(U,2U]^2$ such that $p=p'$ is only $U$, one can show that in this case the above expression is pretty small so is negligible; whilst, Montgomery and Vaughan also show that the sequence $\set{n(p-p')\alpha\mrd{1}}_n$ is equidistributed for almost all pairs $(p,p')$ whenever $\alpha\neq 0$, and this  leads them to a desired bound. It seems that their approach can be generalized to deal with the correlation of $f$ with  polynomial phase functions if $f$ satisfies conditions (\ref{bound-prime}) and (\ref{bound-l2}). In what follows, we provide a heuristic argument for the generalization. Suppose that $k\geq 2$ is a natural number and $0<\delta<1$ is a parameter,  we aim to show that the following inequality holds only if $\alpha$ is $\delta^{-O(1)}$-rational ($\norm{q\alpha}\ll\delta^{-O(1)}/N^k$ for some $1\leq q\ll\delta^{-O(1)}$)
\begin{align*}
\delta\leq\bigabs{\E_{p\sim U}\E_{pn\sim N}\log pf(p)f(n)e(p^kn^k\alpha)}.
\end{align*}
Firstly, it can be seen from H\"older's inequality with some large integer $s=s(k)$ that
\[
\delta\ll\Bigbrac{\E_{n\sim N/U}|f(n)|^{\frac{2s}{2s-1}}}^{\frac{2s-1}{2s}}\Bigbrac{\E_{n}\bigabs{\E_{p\sim U}\log p\,f(p)e(p^kn^k\alpha)}^{2s}}^{\frac{1}{2s}}.
\]
The assumptions $f(p)$ is bounded for every prime $p$ and $\E_{n}|f(n)|^2\ll1$ (or a weaker assumption $\E_{n}|f(n)|^{\frac{2s}{2s-1}}\ll1$)  then lead us to
\[
\delta^{2s}\ll\E_{p_1,\cdots,p_{2s}\sim U}\log p_1\,\cdots\log p_{2s}\,\bigabs{\E_{n\sim N/U}e\bigbrac{(p_1^k+\cdots +p_s^k-\cdots-p_{2s}^k)n^k\alpha}}.
\]
It  follows from \cite[Lemma 1.1.16]{Tao} that there is $q\ll\delta^{-O(1)}$ and a set $\mathcal P\subset(U,2U]^{2s}$ with $|\mathcal P|\gg\delta^{O(1)}U^{2s}$ such that for each  $\vec p=(p_1,\cdots,p_{2s})\in\mathcal P$
\[
\norm{q(p_1^k+\cdots +p_s^k-\cdots-p_{2s}^k)\alpha}\ll\delta^{-O(1)}\bigbrac{\frac{U}{N}}^k.
\]
If we can show that $\set{p_1^k+\cdots +p_s^k-\cdots-p_{2s}^k}_{\vec p\in\mathcal P}$ forms a dense  (up to a factor $\delta^{O(1)}$ ) subset of the interval $[1,O(U^k)]$, the desired conclusion ($\alpha$ is $\delta^{-O(1)}$-rational) would follow from \cite[Lemma 3.4]{GT12b}. Practically,, the fact that $\set{p_1^k+\cdots +p_s^k-\cdots-p_{2s}^k}_{\vec p\in\mathcal P}$ is dense is a consequence of \cite[Lemma 3.3]{GT12b} if $s$ is large compared with the degree $k$. In this paper, however, we are more interested in the correlation of multiplicative functions under weaker conditions (weaker than (\ref{bound-prime}) and (\ref{bound-l2})).

In our understanding, the main result of Jiang, L\"u and Wang \cite{JLW} is that they verify the function $\lambda_\pi$ satisfies conditions (\ref{bound-l2}--\ref{ph-l2}), and then they apply the thought of Montgomery and Vaughan to show that $\lambda_\pi$ doesn't correlate with linear phase functions. Fortunately, from their proof, we find that the function $\lambda_\pi$ has good major-arc behavior, and this motivates us to study the (dis-)correlation of nilsequences with a class of functions that have good major-arc behavior (satisfy condition (\ref{w-equi})).

The first step is to apply Green--Tao factorization theorem, Theorem \ref{factorization}, to decompose the polynomial $g$ as a product $\epsilon g'\gamma$, where $\epsilon$ is a smooth factor, $\gamma$ is rational, and $g'$ is highly equidistributed in a closed subgroup $G'\subseteq G$. We then eliminate the rather harmless factors $\epsilon$ and $\gamma$ just as Green and Tao did in their \cite[Section 2]{GT12b}. This  leaves us to deal with the highly equidistributed sequences with respect to arithmetic progressions with small moduli, say,
\[
\E_P\E_{n\in P}f(n)F_P\bigbrac{g_P(n)\Gamma_P},
\]
where $P$ denotes those arithmetic progressions, and $g_P$ is highly equidistributed for each $P$. The next step is to utilize the approach of Montgomery and Vaughan, to split the function $f(n)$ as a product $f(p)f(n)$, where $p$ is a prime number, so that we are in the position that
\[
\E_P\E_{pn\in P}\log p\,f(p)f(n)F_P\bigbrac{g_P(pn)\Gamma_P}.
\]
Here, we simplify the situation we would actually face, such as $pn$ is in a reduced residue class $pn\equiv b\mrd{W}$. Unlike in the work of Matthiesen \cite{Mat}, our functions $f$ don't have enough information in each progression $P$ which fails us to deal with the summation over $pn$ in progressions $P$ piece by piece. We compromise with taking the above outside sum (counting progression) all the time, and view $F_P$ as a piecewise function which supports on $G/\Gamma$ and has very nice local properties in the sense that  $F_P\bigbrac{g_P(\cdot)\Gamma_P}$ is equidistributed in each progression $P$.

The application of Cauchy-Schwarz inequality allows us to transfer the matter to understanding the equidistribution of product of nilsequences, and the correlation of von Mangoldt function with product of nilsequences,
\[
\E_nF_P(g_P(pn)\Gamma_P)\bar{F_P(g_P(p'n)\Gamma_P)}
\]
and
\[
\E_p\Lambda(p)F_P(g_P(pn)\Gamma_P)\bar{F_P(g_P(pn')\Gamma_P)}.
\]
We then take care of the two expressions by showing that almost all product sequences $\bigbrac{g_P(p\cdot),g_P(p'\cdot)}$ and $\bigbrac{g_P(n\cdot),g_P(n'\cdot)}$ are equidistributed whenever $g_P$ is highly equidistributed, making use of the quantitative Leibman theorem (Proposition \ref{leibman}); and also by using bilinear form method adapted to polynomial nilsequences.

\section{Reducing to Equidistributed Cases}

As described in Section 2, we are going to factorize the polynomial sequence $g$ into a smooth factor $\epsilon$, a rational factor $\gamma$ and a totally equidistributed factor $g'$, then eliminate the rather harmless factors $\epsilon$ and $\gamma$ in the light of  condition (\ref{w-equi}). This will allow us to focus on the highly equidistributed polynomial sequences. The overall strategy of the reduction follows from \cite[Section 2]{GT12b}.

Without loss of generality, we may normalise the Lipschitz function $F$ so that $\norm{F}_\lip=1$, assume that $A>1$ is a large number to be specified, and the parameter $M_0$ in Theorem \ref{main} is equal to the maximal value, that is $M_0\geq\log N$.

 We begin our calculation with applying Proposition \ref{factorization} to the polynomial sequence $g$ to find an integer $M$ with $M_0\leq M\leq M_0^{O_{A}(1)}$; a rational subgroup $G'\subseteq G$; a Mal'cev basis $\mathcal X'$ adapted to $(G')_\bullet$ for which elements are $M$-rational combinations of elements of $\mathcal X$; and a decomposition $g=\epsilon g'\gamma$, where $\bigbrac{\epsilon(n)}_{n\in\Z}$ is $(M,N)$-smooth, $\bigbrac{g'(n)\Gamma'}_{n\in[N]}$ is totally $M^{-A}$-equidistributed in $G'/\Gamma'$, and $\bigbrac{\gamma(n)}_{n\in\Z}$ is periodic with period $q\leq M$.   Immediately, on writing $g$ as the product $\epsilon g'\gamma$, one has,
 \begin{multline*}
\frac{\phi(W)}{WN}\sum_{n\in[N]}\Bigbrac{f(Wn+b)-\E_f(N;W,b)}F(g(n)\Gamma)\\
=\frac{\phi(W)}{WN}\sum_{n\in[N]}\Bigbrac{f(Wn+b)-\E_f(N;W,b)}F\bigbrac{\epsilon(n)g'(n)\gamma(n)\Gamma}.	
\end{multline*}

Since the sequence $\bigbrac{\gamma(n)}_{n\in\Z}$ is periodic with period $q\leq M$, we can partition the discrete interval $[N]$ into arithmetic progressions with common difference $q$ so that the fractional part of $\gamma$ with respect to $\Gamma$ takes constant values in each progression. In practice,  let $P_i$ be the largest progression of the form $i+q\cdot\set{0,1,2,\dots}$ inside of $[N]$ and of common difference $q$ and starting point $i=\set{0,1,\dots,q-1}$ respectively. Then for every integer $0\leq i<q$, there is a number $\gamma_i$ such that $\gamma_i\Gamma=\gamma(P_i)\Gamma$; besides, from \cite[Section 2]{GT12b}, all the coordinates $\psi_\mathcal X(\gamma_i)$ are rationals with height at most $O(M^{O(1)})$. Therefore, we see quickly  that,
\begin{multline*}
\frac{\phi(W)}{WN}\sum_{n\in[N]}\Bigbrac{f(Wn+b)-\E_f(N;W,b)}F\bigbrac{\epsilon(n)g'(n)\gamma(n)\Gamma}
\\=\frac{\phi(W)}{WN}\sum_{0\leq i<q}\sum_{n\in P_i}\Bigbrac{f(Wn+b)-\E_f(N;W,b)}F\bigbrac{\epsilon(n)g'(n)\gamma_i\Gamma}.	
\end{multline*}

In the next step, we'd like to split $P_i$ further into sub-progressions such that $\epsilon$ is approximately constant in each of these sub-progressions. For this purpose, we fix a number $n_0\in[N]$ for the moment and consider the difference
\[
\Bigabs{F\bigbrac{\epsilon(n_0)g'(n)\gamma_i\Gamma}-F\bigbrac{\epsilon(n)g'(n)\gamma_i\Gamma}}.
\]
Taking note that the function $F: G/\Gamma\to\C$ has Lipschitz norm 1, one can deduce from the definition of Lipschitz norm that
\[
\Bigabs{F\bigbrac{\epsilon(n_0)g'(n)\gamma_i\Gamma}-F\bigbrac{\epsilon(n)g'(n)\gamma_i\Gamma}}\leq\rd_\mathcal X\bigbrac{\epsilon(n_0)g'(n)\gamma_i,\epsilon(n)g'(n)\gamma_i}.
\]
From the perspective of the operator $\rd_\mathcal X$ is right-invariant, one has
\[
\rd_\mathcal X\bigbrac{\epsilon(n_0)g'(n)\gamma_i,\epsilon(n)g'(n)\gamma_i}=\rd_\mathcal X\bigbrac{\epsilon(n_0),\epsilon(n)}.
\]
The assumption $\epsilon$ is $(M,N)$-smooth then yields that
\[
\rd_\mathcal X\bigbrac{\epsilon(n_0),\epsilon(n)}\leq\abs{n-n_0}\sup_{m\in[n_0,n]}\rd_\mathcal X\bigbrac{\epsilon(m-1),\epsilon(m)}\leq\frac{M|n-n_0|}{N}.
\]
Thus, we can conclude that
\begin{align}\label{diff}
\Bigabs{F\bigbrac{\epsilon(n_0)g'(n)\gamma_i\Gamma}-F\bigbrac{\epsilon(n)g'(n)\gamma_i\Gamma}}\ll\frac{1}{\log N},
\end{align}
whenever $|n-n_0|\leq\frac{N}{M\log N}$. Taking the advantage of above analysis, one can decompose $P_i$ into sub-progressions $P_{i,j}$ in which each $P_{i,j}$ has diameter length at most $O(\frac{N}{M\log N})$, and as a consequence,  there are at most $O(M\log N)$ disjoint progressions. By fixing an element $\epsilon_{i,j}\in\bigset{\epsilon(P_{i,j})}$ for each progression $P_{i,j}$, plainly, we have
\begin{multline*}
\frac{\phi(W)}{WN}\sum_{0\leq i<q}\sum_{n\in P_i}\Bigbrac{f(Wn+b)-\E_f(N;W,b)}F\bigbrac{\epsilon(n)g'(n)\gamma_i\Gamma}\\
=\frac{\phi(W)}{WN}\sum_{i,j}\sum_{n\in P_{i,j}}\Bigbrac{f(Wn+b)-\E_f(N;W,b)}F\bigbrac{\epsilon_{i,j}g'(n)\gamma_i\Gamma}	\\
+\frac{\phi(W)}{WN}\sum_{i,j}\sum_{n\in P_{i,j}}\Bigbrac{f(Wn+b)-\E_f(N;W,b)}\Bigset{F\bigbrac{\epsilon(n)g'(n)\gamma_i\Gamma}-F\bigbrac{\epsilon_{i,j}g'(n)\gamma_i\Gamma}}.
\end{multline*}
Since it follows from Cauchy-Schwarz inequality and condition (\ref{wl2}) that
\begin{align}\label{wl1}
\frac{\phi(W)}{WN}\sum_{n\leq N}|f(Wn+b')|\ll N^{-1}\Bigbrac{\frac{\phi(W)}{W}}^{1/2}\Bigbrac{\frac{\phi(W)}{WN}\sum_{n\in[N]}|f(Wn+b)|^2}^{1/2}\ll 1	
\end{align}
holds for every integer $1\leq b'\leq W$ coprime with $W$.
We are able to bound the above second term as follows due to  inequality (\ref{diff}) and (\ref{wl1})
\begin{multline*}
\frac{\phi(W)}{WN}\sum_{i,j}	\sup_{n\in P_{i,j}}\Bigabs{F\bigbrac{\epsilon_{i,j}g'(n)\gamma_i\Gamma}-F\bigbrac{\epsilon(n)g'(n)\gamma_i\Gamma}}\sum_{n\in P_{i,j}}\Bigbrac{|f(Wn+b)|+|\E_f(N;W,b)|}\\
\ll\frac{1}{\log N}\frac{\phi(W)}{WN}\sum_{n\in[N]}|f(Wn+b)|
\ll \frac{1}{\log N},\qquad\qquad\qquad
\end{multline*}
thus above term can be absorbed by the error term of Theorem \ref{main}.
Besides, it also follows from the assumption $\epsilon$ is $(M,N)$-smooth that $\rd_\mathcal X(\epsilon_{i,j},\mathrm{id}_G)\leq M$ and thus, by \cite[Lemma A.4]{GT12a}, $\psi_\mathcal X(\epsilon_{i,j})\ll M^{O(1)}$.

It leaves us to concentrate on settling the above first term. When the totally equidistributed part $\bigbrac{g'(n)\Gamma'}_{n\in[N]}$ is trivial, i.e. $g'(n)=\text{id}_G\mod\Gamma$, it becomes
\[
\frac{\phi(W)}{WN}\sum_{0\leq i<q}\sum_{n\in P_i}c_{i,j}\bigbrac{f(Wn+b)-\E_f(N;W,b)}
\]
for some $c_{i,j}\in\C$ with $|c_{i,j}|\leq1$, in view of $F:G/\Gamma\to\C$ is a 1-bounded function. It then can be easily verified from  condition (\ref{w-equi})  that this term is bounded by $O\bigbrac{\frac{1}{\log N}}$, and the theorem follows. Hence, in the rest part of this paper we would always assume that $\bigbrac{g'(n)\Gamma'}_{n\in[N]}$ isn't trivial.

For a fixed pair $(i,j)$, write $(H_i)_\bullet=\gamma_i^{-1}G'_\bullet\gamma_i$, and let $g_i\in\poly(\Z,(H_i)_\bullet)$ be the polynomial sequence defined via $g_i(n)=\gamma_i^{-1}g'(n)\gamma_i$. Besides, suppose that $\Lambda_i=\Gamma\cap H_i$, and define the 1-bounded Lipschitz function $F_{i,j}:H_i/\Lambda_i\to\C$ in the manner of $F_{i,j}(x\Lambda_i)=F(\epsilon_{i,j}\gamma_ix\Gamma)$. Thus, for each pair $(i,j)$,
\begin{multline*}
\sum_{n\in P_{i,j}}\Bigbrac{f(Wn+b)-\E_f(N;W,b)}F\bigbrac{\epsilon_{i,j}g'(n)\gamma_i\Gamma}\\
=\sum_{n\in P_{i,j}}\Bigbrac{f(Wn+b)-\E_f(N;W,b)}F_{i,j}\bigbrac{g_i(n)\Lambda_i}.	
\end{multline*}
Now, assume that $\mu_{i,j}=\int_{H_i/\Lambda_i}F_{i,j}$ as the mean value of $F_{i,j}$ on the nilmanifold $H_i/\Lambda_i$ and then rewrite $F_{i,j}$ as $\brac{F_{i,j}-\mu_{i,j}}+\mu_{i,j}$. Clearly, $F_{i,j}-\mu_{i,j}$ is a bounded Lipschitz function and $\int_{H_i/\Lambda_i}(F_{i,j}-\mu_{i,j})=0$. Moreover, taking note that $P_{i,j}$ are progressions of size $|P_{i,j}|\gg\frac{N}{qM\log N}\gg\frac{N}{qM^2}$, it can be deduced from  condition (\ref{w-equi}), as well as $|\mu_{i,j}|\ll1$ for all pairs $(i,j)$, that
\begin{multline*}
\biggabs{\frac{\phi(W)}{WN}\sum_{i,j}\sum_{n\in P_{i,j}}\mu_{i,j}\Bigbrac{f(Wn+b)-\E_f(N;W,b)}}	\\\ll\frac{\phi(W)}{WN}\sum_{i,j}\Bigabs{\sum_{n\in P_{i,j}}f(Wn+b)-\E_f(N;W,b)}\ll\frac{1}{\log N}.
\end{multline*}
Thus, without loss of generality, we may assume that $F_{i,j}:H_i/\Lambda_i\to\C$ is a bounded Lipschitz function with $\int_{H_i/\Lambda_i}F_{i,j}=0$. Just sum up what we've agreed so far, we are in the position that
\begin{align}\label{finalre}
&\frac{\phi(W)}{WN}\sum_{n\in[N]}\Bigbrac{f(Wn+b)-\E_f(N;W,b)}F(g(n)\Gamma)\nonumber\\
=&\frac{\phi(W)}{WN}\sum_{i,j}\sum_{n\in P_{i,j}}\Bigbrac{f(Wn+b)-\E_f(N;W,b)}F_{i,j}\bigbrac{g_i(n)\Lambda_i}+O(\frac{1}{\log N}).
\end{align}
With the help of \cite[Claim in Section 2]{GT12b}, together with $\psi_\mathcal X(\gamma_i),\psi_\mathcal X(\epsilon_{i,j})\ll M^{O(1)}$, one may find a Mal'cev basis $\mathcal Y_i$ for $H_i/\Lambda_i$ adapted to $(H_i)_\bullet$ for which each $\mathcal Y_i$ is an $M^{O(1)}$-rational combination of elements of $\mathcal X$; besides, we also have $\norm{F_{i,j}}_\lip\ll M^{O(1)}$ and $g_i\in\poly(\Z,(H_i)_\bullet)$ with $\bigbrac{g_i(n)}_{n\in[N]}$ is totally $M^{-cA+O(1)}$-equidistributed for some constant $c>0$.
Besides, it can be seen from  (\ref{wl1}) that
\[
\E_f(N;W,b)=\frac{\phi(W)}{WN}\sum_{n\in[N]}f(Wn+b)\ll1.	
\]
Therefore, using the fact that $(g_i(n))_{n\in[N]}$ is totally  $M^{-cA+{O(1)}}$-equidistributed, and noting that $\norm{F_{i,j}}_{\lip}\ll M^{-O(1)}$ and $\int_{H_i/\Lambda_i}F_{i,j}=0$, we can get
\begin{multline*}
\frac{\phi(W)}{WN}\sum_{i,j}|\E_f(N;W,b)|\cdot\Bigabs{\sum_{n\in P_{i,j}}F_{i,j}\bigbrac{g_i(n)\Lambda_i}}\ll\frac{\phi(W)}{WN}\sum_{i,j}\Bigabs{\sum_{n\in P_{i,j}}F_{i,j}\bigbrac{g_i(n)\Lambda_i}}\\
\ll M^{-cA+O(1)}\frac{\phi(W)}{WN}\sum_{i,j}|P_{i,j}|\ll M^{-cA+O(1)}.
\end{multline*}

In view of (\ref{finalre}) and above  inequality, we see that to prove Theorem \ref{main} it suffices to show that
\[
\frac{\phi(W)}{WN}\sum_{i,j}\sum_{n\in P_{i,j}}f(Wn+b)F_{i,j}\bigbrac{g_i(n)\Lambda_i}\ll \frac{1}{\log N}.
\]
And this is the main business of the next two sections.

\section{Applying the Approach of Montgomery and Vaughan}\label{apply-mv}

For clarity and completeness, we restate our task here again. Let $A>1$ be a large number and $\log N\leq M\leq (\log N)^{O_A(1)}$. Let $G'\subseteq G$ be a subgroup of $G$, $g'\in\poly(\Z,G_\bullet)$ and $\bigbrac{g'(n)}_{n\in[N]}$ be a totally $M^{-cA}$-equidistributed sequence for some constant $c>0$. Suppose that $P_{i,j}$ are pairwise disjoint arithmetic progressions with common difference $q\leq M$ and of length $|P_{i,j}|\geq\frac{N}{qM^2}$, and $\sqcup_{i,j}P_{i,j}=[N]$. Suppose that $(\epsilon_{i,j})_{i,j}$ and $(\gamma_i)_i$ are sequences of $G$. Set $H_i=\gamma_i^{-1}G'\gamma_i$, $\Lambda_i=\Gamma\cap H_i$, and $\mathcal Y_i$ is $M^{O(1)}$-rational Mal'cev basis adapted to filtration $(H_i)_{\bullet}$. Suppose that  $F_{i,j}(x\Lambda_i)=F(\epsilon_{i,j}\gamma_ix\Gamma)$ with $\norm{F_{i,j}}_{\lip}\ll M^{-O(1)}$ and $\int_{H_i/\Lambda_i}F_{i,j}=0$, and $\bigbrac{g_i(n)\Lambda_i}_{n\in[N]}$ is totally $M^{-A}$-equidistributed with $g_i=\gamma_i^{-1}g'\gamma_i$. We'll focus on handling the expression
\[
\frac{\phi(W)}{WN}\sum_{i,j}\sum_{n\in P_{i,j}}f(Wn+b)F_{i,j}(g_i(n)\Lambda_i),
\]
where $f$ is an arbitrary function from $\mathcal M'$.

On employing the approach of Montgomery and Vaughan, we'd like to add a logarithmic factor into above expression, and the aim of this section is to transfer the question to understand the equidistribution matter along with variables in terms of the prime numbers.

\begin{lemma}\label{mv}
Suppose that $H_i/\Lambda_i$ are nilmanifolds, and that for each $i$ there is a $M^{O(1)}$-rational Mal'cev basis $\mathcal Y_i$  adapted to filtration $(H_i)_{\bullet}$, where  $\log N\leq M\leq (\log N)^C$. Suppose that  $P_{i,j}$ are disjoint arithmetic progressions with common difference $q\leq M$ and length $\Omega(\frac{N}{qM^2})$, and the union of  $P_{i,j}$ is the discrete interval $[N]$. Assume that $g_i\in\poly(\Z,(H_i)_\bullet)$ are polynomial sequences and $F_{i,j}:H_i/\Lambda_i\to\C$ are bounded functions, then for any multiplicative function $f\in\mathcal M'$, we have
\begin{multline*}
\frac{\phi(W)}{WN}\sum_{i,j}\sum_{n\in P_{i,j}}f(Wn+b)F_{i,j}(g_i(n)\Lambda_i)\\
\ll\frac{1}{\log N}\biggabs{\frac{\phi(W)}{WN}\sum_{i,j}\sum_{pn\in W\cdot P_{i,j}+b}\log p\,f(p)f(n)F_{i,j}\Bigbrac{g_i\bigbrac{\frac{pn-b}{W}}\Lambda_i}}+\frac{1}{\log N}.	
\end{multline*}

\end{lemma}

\begin{proof}
 We start the proof by adding a suitable log-factor into the target expression as follows
\[
\frac{\phi(W)}{WN}\sum_{i,j}\sum_{n\in P_{i,j}}f(Wn+b)F_{i,j}(g_i(n)\Lambda_i)\log\frac{WN+b}{Wn+b}.
\]
Directly, on the one hand, it equals to
\[
\frac{\phi(W)}{WN}\sum_{i,j}\sum_{n\in P_{i,j}}\Bigset{\log(WN+b)-\log(Wn+b)}f(Wn+b)F_{i,j}(g_i(n)\Lambda_i).	
\]
On the other hand, from the assumption that $F_{i,j}$ are bounded functions and all progressions $\set{P_{i,j}}_{i,j}$ form a partition of $[N]$,  making use of Cauchy-Schwarz inequality and condition (\ref{wl2}) gives that
\begin{multline*}
\frac{\phi(W)}{WN}\sum_{i,j}\sum_{n\in P_{i,j}}f(Wn+b)F_{i,j}(g_i(n)\Lambda_i)\log\frac{WN+b}{Wn+b}\\
\leq\biggbrac{\frac{\phi(W)}{WN}\sum_{n\in[N]}|f(Wn+b)|^2}^{1/2}\biggbrac{\frac{\phi(W)}{WN}\sum_{n\in[N]}\bigbrac{\log(WN+b)-\log(Wn+b)}^2}^{1/2}\\
\ll	\biggbrac{\frac{\phi(W)}{WN}\sum_{n\in[N]}\bigbrac{\log^2 (N+\frac{b}{W})-2\log (N+\frac{b}{W})\log (n+\frac{b}{W})+\log^2(n+\frac{b}{W})}}^{1/2}\ll1.
\end{multline*}
Combining the above two estimates it turns out that
\begin{multline*}
\frac{\phi(W)}{WN}\sum_{i,j}\sum_{n\in P_{i,j}}f(Wn+b)F_{i,j}(g_i(n)\Lambda_i)\\
\ll\frac{1}{\log N}\biggabs{\frac{\phi(W)}{WN}\sum_{i,j}\sum_{n\in P_{i,j}}\log(Wn+b)f(Wn+b)F_{i,j}(g_i(n)\Lambda_i)}+\frac{1}{\log N}.	
\end{multline*}

By denoting $x=Wn+b$ with $n\in P_{i,j}$, one has $x\in W\cdot P_{i,j}+b$, and then rename $x$ as $n$ to get that
\[
\sum_{n\in P_{i,j}}\log(Wn+b)f(Wn+b)F_{i,j}\bigbrac{g_i(n)\Lambda_i}=\sum_{n\in W\cdot P_{i,j}+b}\log n\,f(n)F_{i,j}\Bigbrac{g_i\bigbrac{\frac{n-b}{W}}\Lambda_i}
\]

The formula $\log n=\sum_{m|n}\Lambda(m)$, which can be viewed as a generation of the fundamental theorem of arithmetic, then yields that the above identity is also equal to
\[
\sum_{mn\in W\cdot P_{i,j}+b}\Lambda(m)f(mn)F_{i,j}\Bigbrac{g_i\bigbrac{\frac{mn-b}{W}}\Lambda_i}.
\]
In light of $\Lambda(m)$ is zero unless $m$ is a prime power, we may rewrite the above expression as
\begin{multline*}
\sum_{pn\in W\cdot P_{i,j}+b}\log p\,f(p)f(n)F_{i,j}\Bigbrac{g_i\bigbrac{\frac{pn-b}{W}}\Lambda_i}\\
+\sum_{pn\in W\cdot P_{i,j}+b}\log p\,\bigset{f(pn)-f(p)f(n)}F_{i,j}\Bigbrac{g_i\bigbrac{\frac{pn-b}{W}}\Lambda_i}\\
+\sum_{p,k\geq2}\sum_{p^kn\in W\cdot P_{i,j}+b}\log p\,f(p^kn)F_{i,j}\Bigbrac{g_i\bigbrac{\frac{p^kn-b}{W}}\Lambda_i}.
\end{multline*}
We then quickly see that this lemma follows if we can prove that the contribution from the collection of all progressions $P_{i,j}$ in terms of the above latter two  terms times $\frac{\phi(W)}{WN}$ is bounded by 1.

Since $f$ is multiplicative, $f(pn)-f(p)f(n)$ vanishes unless $p$ is a divisor of $n$. On recalling the facts that $F_{i,j}$ is a bounded function and the collection of $P_{i,j}$ forms a partition of $[N]$, one has
\begin{multline*}
\biggabs{\frac{\phi(W)}{WN}\sum_{i,j}\twosum{pn\in W\cdot P_{i,j}+b}{p|n}\log p\,\bigset{f(pn)-f(p)f(n)}F_{i,j}\Bigbrac{g_i\bigbrac{\frac{pn-b}{W}}\Lambda_i}}\\
\ll\frac{\phi(W)}{WN}\sum_{p,k\geq2}\twosum{p^kn\in[WN+b]}{p^kn\equiv b\mrd{W}}\log p\,\Bigset{|f(p^k)||f(n)|+|f(p)||f(p^{k-1})||f(n)|}.	
\end{multline*}
 We deal with above first term at first. As $(b,W)=1$ and $1\leq b\leq W\ll(\log N)^C$ for some constant $C>0$, so that $b$ is quite smaller than $N$, and thus, the first summation term is bounded by
 \[
 \frac{\phi(W)}{WN}\twosum{p,k\geq2}{(p,W)=1}\log p\,|f(p^k)|\twosum{n\leq WN/p^k}{n\equiv b\bar{p^k}\mrd{W}}|f(n)|\ll\twosum{p,k\geq2}{(p,W)=1}\frac{\log p\,|f(p^k)|}{p^k},
 \]
where we have applied the following inequality (\ref{wl1}) to the inner summation.

On applying Cauchy-Schwarz inequality another time,  one can deduce from condition (\ref{lp2}) that
\[
\twosum{p,k\geq2}{(p,W)=1}\frac{\log p\,|f(p^k)|}{p^k}\leq\Bigbrac{\sum_{p,k\geq2}\frac{(\log p)^2}{p^{3k/4}}}^{1/2}\Bigbrac{\sum_{p,k\geq2}\frac{|f(p^k)|^2}{p^{5k/4}}}^{1/2}\ll1.
\]
Besides, there is no difficulty to find that in the same manner we can also have
\[
\biggabs{\frac{\phi(W)}{WN}\sum_{i,j}\sum_{p,k\geq2}\sum_{p^kn\in W\cdot P_{i,j}+b}\log p\,f(p^kn)F_{i,j}\Bigbrac{g_i\bigbrac{\frac{p^kn-b}{W}}\Lambda_i}}\ll1.
\]
It, therefore, remains to calculate the following
\begin{multline*}
\frac{\phi(W)}{WN}\twosum{p,k\geq2}{(p,W)=1}\log p\,|f(p)||f(p^{k-1})|\twosum{n\leq WN/p^k}{n\equiv b\bar{p^k}\mrd{W}}|f(n)|\\
\ll\sum_{p,k\geq 1}\frac{|f(p)|}{p^{1/3}}\frac{|f(p^k)|}{p^{k/3}}\log p\,p^{-2/3}p^{-2k/3}.
\end{multline*}
Using the inequality $ab\leq|a|^2+|b|^2$, together with condition (\ref{lp2}), it is bounded by
\[
\sum_{p\geq 1}\frac{|f(p)|^2}{p^{4/3}}\log p\sum_{k\geq 1}p^{-2k/3}+\sum_{p,k\geq1}\frac{|f(p^k)|^2}{p^{4k/3}}\log p\, p^{-2/3}\ll1	.
\]

We then finish the proof of this lemma.

\end{proof}

The next step is to decompose the summation range of prime $p$ into two domains, according to whether it is close to $N$ or not. Let
\begin{align}\label{uv}
U=N^{2/3}	
\end{align}
be a cutoff parameter, it is clear that,
\begin{multline}\label{decom}
\frac{\phi(W)}{WN}\sum_{i,j}\sum_{pn\in W\cdot P_{i,j}+b}\log p\,f(p)f(n)F_{i,j}\Bigbrac{g_i\bigbrac{\frac{pn-b}{W}}\Lambda_i}\\
=\frac{\phi(W)}{WN}\sum_{i,j}\biggset{\sum_{p\leq U}+\sum_{U<p\leq N}}\sum_{pn\in W\cdot P_{i,j}+b}\log p\,f(p)f(n)F_{i,j}\Bigbrac{g_i\bigbrac{\frac{pn-b}{W}}\Lambda_i}.	
\end{multline}
And we are going to handle the above two summations respectively in the next section. The proof idea can be explained as follows. Suppose that $g$ is a highly equidistributed polynomial,  when $p$ is small, then $g_p=g(p\,\cdot)$ would also look like equidistributed and thus the summation $\sum_n F\bigbrac{g_p(n)\Gamma}$ should have some cancellation; whilst when $p$ is large, since $pn$ is no more than $WN$ with $W\ll(\log N)^C$, $n$ should be small, and in this case the sum $\sum_p F\bigbrac{g_n(p)\Gamma}$ also has some cancellation.

\section{Equidistribution of Product Nilsequences}

In this section,  we'll handle the two representations in (\ref{decom}) in turns.

\subsection{When $p$ is small}

\begin{lemma}\label{pissmall}
\[
\frac{\phi(W)}{WN}\sum_{i,j}\sum_{p\leq U}\sum_{pn\in W\cdot P_{i,j}+b}\log p\,f(p)f(n)F_{i,j}\Bigbrac{g_i\bigbrac{\frac{pn-b}{W}}\Lambda_i}\ll1,
\]
where all of the parameters are in Section \ref{apply-mv}.	
\end{lemma}

\begin{proof}

 To begin with, we split $p\leq U$ into dyadic ranges $p\sim N_k$ with $N_k=2^{-k}U$ and $0\leq k\leq\frac{2\log N}{3\log 2}$. Let $P_{i,j}=a_{i,j}+q\cdot[X]$, where $a_{i,j}+q$ is the starting point of the progression $P_{i,j}$ and, by assumption, we have $X\geq\frac{N}{qM^2}$. After decomposing the sum in the congruence condition $pn\equiv Wa_{i,j}+b\, (\mathrm{mod}\, Wq)$ further, we are in the position that
\begin{align}\label{total}
\frac{\phi(W)}{WN}\sum_{i,j,k}\twosum{a_1,a_2\mrd{Wq}}{a_1a_2\equiv Wa_{i,j}+b\mrd{Wq}}\twosum{p\sim N_k}{p\equiv a_1\mrd{Wq}}\sum_{n}\log p\,f(p)f(n) F_{i,j}\Bigbrac{g_i\bigbrac{\frac{pn-b}{W}}\Lambda_i},
\end{align}
where $n$ ranges over $\frac{Wa_{i,j}+b}{p}\leq n\leq\frac{WqX+Wa_{i,j}+b}{p}$ and obeys the congruence condtion $n\equiv a_2\mrd{Wq}$.
Switching the summation order of variables $p$ and $n$, and using Cauchy-Schwarz inequality, one has
\begin{multline}\label{pn}
\twosum{p\sim N_k}{p\equiv a_1\mrd{Wq}}\twosum{\frac{Wa_{i,j}+b}{p}\leq n\leq\frac{WqX+Wa_{i,j}+b}{p}}{n\equiv a_2\mrd{Wq}}\log p\,f(p)f(n) F_{i,j}\Bigbrac{g_i\bigbrac{\frac{pn-b}{W}}\Lambda_i}\\
\leq\biggbrac{\sum_{n}|f(n)|^2}^{1/2}
\biggbrac{\twosum{n}{n\equiv a_2\mrd{Wq}}\biggabs{\twosum{p\sim N_k}{p\equiv a_1\mrd{Wq}}\log p\,f(p)F_{i,j}\Bigbrac{g_i\bigbrac{\frac{pn-b}{W}}\Lambda_i}}^2}^{1/2}	,
\end{multline}
where $n$ ranges over the interval $\frac{Wa_{i,j}+b}{N_k}\leq n\leq\frac{WqX+Wa_{i,j}+b}{N_k}$. The first factor is easily to calculate. Indeed, condition (\ref{fl2}) gives that
\begin{align}\label{f2}
\biggbrac{\sum_{n\leq\frac{WqX+Wa_{i,j}+b}{N_k}}|f(n)|^2-\sum_{n\leq\frac{Wa_{i,j}+b}{N_k}}|f(n)|^2}^{1/2}\ll\bigbrac{\frac{WqX}{N_k}}^{1/2}.
\end{align}
For the second factor, expanding the square to double the variables $p$  and then swapping the order of summation to get that
\begin{align}\label{pp'}
\biggbrac{\twosum{p,p'\sim N_k}{p\equiv p'\equiv a_1\mrd{Wq}}\log p\log p'f(p)f(p')\sum_{n}F_{i,j}\Bigbrac{g_i\bigbrac{\frac{pn-b}{W}}\Lambda_i}\bar{F_{i,j}\Bigbrac{g_i\bigbrac{\frac{p'n-b}{W}}\Lambda_i}}}^{1/2},
\end{align}
where $n$ ranges over $Wa_{i,j}+b\leq pn,p'n\leq WqX+Wa_{i,j}+b$ and is also in the residue class $n\equiv a_2\mrd{Wq}$.

 Let $n=a_2+Wqm$, taking note that $p\equiv a_1\mrd{Wq}$, $a_1a_2\equiv b\mrd{Wq}$, as well as $1\leq a_1\leq Wq$, we can conclude that there must be integers $|a_p|,|a_{p'}|\leq qN_k$ such that $\frac{pn-b}{W}=qpm+a_p$ and $\frac{p'n-b}{W}=qp'm+a_{p'}$ respecitvely. Thus, above inner summation over $n$ is indeed
 \begin{align}\label{fij}
 \sum_{m\in \frac{I}{\max\set{p,p'}}}F_{i,j}\bigbrac{g_i(qpm+a_p)\Lambda_i}\bar{F_{i,j}\bigbrac{g_i(qp'm+a_{p'})\Lambda_i}},
 \end{align}
where $I\subseteq[N/q]$ is a discrete subinterval of length at least $X$.

Thanks to the sequences $g_i$ are totally $M^{-A}$-equidistributed, there is a chance to show that for almost all  pairs of primes $(p,p')\in(N_k,2N_k]^2$ the multiplicative sequence $\bigbrac{g_i(qp\cdot+a_p),g_i(qp'\cdot+a_{p'})}\in\poly(\Z,(H_i)_\bullet\times (H_i)_\bullet)$ is equidistributed whenever this sequence is not too short. The following result is similar to \cite[Proposition 8.1]{Mat}. And we hope the proof could serve as an easy-reading exposition of \cite[Proposition 8.1]{Mat}.

\begin{lemma}[Equidistribution of product nilsequences]\label{prose}
Suppose that $0<\delta<1/2$ is a number. There are two small numbers $0<c<c'<1$ such that the following statement holds uniformly for integers $K$ with $1\leq K\leq \delta^{c'}N$.

Suppose that $H_\bullet$ is a filtration of finite degree and of a finite dimensional nilmanifold $H/\Lambda$. Suppose that $g\in\poly(\Z,H_\bullet)$ is a polynomial  and the sequence $(g(n)\Lambda)_{n\in[N]}$ is totally $\delta$-equidistributed. Suppose that $q\ll\delta^{-O(1)}$, $a_p$ and $a_{p'}$ are integers satisfying $|a_p|,|a_{p'}|\leq qK$, and $I\subseteq[N/q]$ is an interval of length $|I|\geq\delta^{O(1)}\frac{N}{q}$. Assume further that $F:H/\Lambda\to\C$ is a Lipschitz function with $\int_{H/\Lambda}F=0$.

Write $\mathcal E_K$ as the set of pairs $(p,p')\in(K,2K]^2$ for which
\[
\#\bigset{m:m\in \frac{I}{\max\set{p,p'}}}\gg\delta^{O(c)}N/K,
\]
and
\[
\Bigabs{\sum_{m\in \frac{I}{\max\set{p,p'}}}F(g(qpm+a_p)\Lambda)\bar{F(g(qp'm+a_{p'})\Lambda)}}>(1+\norm{F}_\lip)\delta^{O(c)} N/K.
\]
Then we have
\[
\#\mathcal E_K\ll\delta^{O(c)}K^2.
\]

\begin{proof}
In the following, we only consider those pairs $(p,p')\in(K,2K]^2$ such that
\[
\#\bigset{m:m\in \frac{I}{\max\set{p,p'}}}\gg\delta^{O(c)}N/K.
\]
Now assume for contradiction that $\#\mathcal E_K\gg\delta^{O(c)}K^2$, which means that there are $\gg\delta^{O(c)}K^2$ pairs of primes $(p,p')\in(K,2K]^2$ such that
\[
\Bigabs{\sum_{m\in \frac{I}{\max\set{p,p'}}}F(g(qpm+a_p)\Lambda)\bar{F(g(qp'm+a_{p'})\Lambda)}}>(1+\norm{F}_\lip)\delta^{O(c)} N/K.
\]
For such a pair $(p,p')\in(K,2K]^2$, we define  $g_{p,p'}(n)=\bigbrac{g(qpn+a_p),g(qp'n+a_{p'})}$ as a new polynomial. Then we have $g_{p,p'}\in\poly\bigbrac{\Z,H_\bullet\times H_\bullet}$. Besides, set a Lipchitz function $\tilde F:H\times H\to\C$ via $\tilde F(\gamma,\gamma')=F(\gamma)\bar{F(\gamma)}$, then $\int_{H/\Lambda\times H/\Lambda}\tilde F=0$ and $\|\tilde F\|_\lip\ll\norm{F}_\lip$. Moreover, the above inequality can be rewritten as --- there are $\gg\delta^{O(c)}K^2$  pairs of $(p,p')\in(K,2K]^2$ such that
\[
\biggabs{\sum_{m\in \frac{I}{\max\set{p,p'}}}\tilde F(g_{p,p'}(m)\Lambda\times\Lambda)}\gg(1+\|\tilde F\|_\lip)\delta^{O(c)} N/K.
\]

Clearly, it can be deduced from Definition \ref{almost-equidistribution} and the assumption $\#\bigset{m:m\in \frac{I}{\max\set{p,p'}}}\gg\delta^{O(c)}N/K$ that there are $\gg\delta^{O(c)}K^2$ of $(p,p')\in(K,2K]^2$ such that the corresponding polynomial sequences $\bigbrac{g_{p,p'}(n)}_{n\in\frac{I}{\max\set{p,p'}}}$ fails to be $\delta^{O(c)}$-equidistributed. It then follows from Theorem \ref{leibman} that for each pair of primes $(p,p')$ there is a nontrivial horizontal character $0<|\psi_{p,p'}|\ll\delta^{-O(c)}$ such that
\[
\norm{\psi_{p,p'}\circ g_{p,p'}}_{C^\infty\frac{I}{\max\set{p,p'}}}\ll\delta^{O(c)}.
\]
On the other hand, for any nontrivial horizontal character $\psi:H\times H\to\T$ of modulus $\ll\delta^{-O(c)}$, define a set
\[
S_\psi=\set{(p,p')\in(K,2K]^2:\psi_{p,p'}=\psi}.
\]
Since the total number of characters $\psi_{p,p'}$ is at least $\Omega(\delta^{O(c)}K^2)$, it can be deduced from pigeonhole principle that there is a nontrivial horizontal character $0<|\psi|\ll\delta^{-O(c)}$ such that $|S_{\psi}|\gg \delta^{O(c)}K^2$. Fixing a character $\psi$ with $| S_{\psi}|\gg \delta^{O(c)}K^2$, we thus have
\[
\norm{\psi\circ g_{p,p'}}_{C^\infty\frac{I}{\max\set{p,p'}}}\ll\delta^{O(c)}
\]
holds for $\gg\delta^{O(c)}K^2$ pairs of $(p,p')\in(K,2K]^2$. Taking $\psi=\psi_1\oplus\psi_2$ where $\psi_1,\psi_2:H\to\T$ and $\psi_1$ is non-trivial, and let
\begin{eqnarray*}
(\psi_1\circ g)(n)=\alpha_dn^d+\dots+\alpha_1n+\alpha_0;\\
(\psi_2\circ g)(n)=\alpha_d'n^d+\dots+\alpha_1'n+\alpha_0'.	
\end{eqnarray*}
Hence,
\begin{multline*}
(\psi\circ g_{p,p'})(n)=\alpha_d(qpn+a_p)^d+\dots+\alpha_0+\alpha_d'(qp'n+a_{p'})^d+\dots+\alpha_0'\\
=\sum_{1\leq j\leq d}n^j\sum_{j\leq i\leq d}\binom i jq^j\bigbrac{p^ja_p^{i-j}\alpha_i+p'^ja_{p'}^{i-j}\alpha_i'}+\tilde\alpha_0.	
\end{multline*}
From the definition of smoothness norm in Definition \ref{smooth-norm}, one may find that there are $\gg\delta^{O(c)}K^2$ of pairs $(p,p')\in(K,2K]^2$ such that
\begin{align}\label{induction}
\norm{q^j\bigbrac{p^ja_p^{i-j}\alpha_i+p'^ja_{p'}^{i-j}\alpha_i'}}\ll\delta^{-O(c)}\Bigbrac{\frac{K}{N}}^j	
\end{align}
for all $1\leq j\leq d$, as $\#\bigset{m:m\in \frac{I}{\max\set{p,p'}}}\gg\delta^{O(c)}N/K$. We claim that there is a non-zero integer $0<q'\ll\delta^{-O(c)}$ such that
\begin{align}\label{claim}
\norm{q'\alpha_j}\ll\delta^{-O(c)}N^{-j}\qquad\text{for all } 1\leq j\leq d.	
\end{align}
When $j=d$, from (\ref{induction}), there are $\gg\delta^{O(c)}K^2$ of pairs $(p,p')\in(K,2K]^2$ such that $\norm{q^dp^d\alpha_d+q^dp'^d\alpha_d'}\ll\delta^{-O(c)}\bigbrac{\frac{K}{N}}^{d}$. By pigeonhole principle there is some $p'\sim K$ such that   at least  $\gg\delta^{O(c)}K$ of $p\sim K$ satisfying the sequence $\set{q^dp^d\alpha\mrd{\Z}}$ stays in an interval of length $O(\delta^{-O(c)}\bigbrac{\frac{K}{N}}^{d})$. Here we assume that  $\delta^{-O(c)}\bigbrac{\frac{K}{N}}^{d}\ll\delta^{O(1)}$, this can be done by taking $0<c'<1$ larger than $c$, recalling that $K\leq\delta^{O(c')}N$ (we'll find that $c$ is sufficiently small). The application of \cite[Lemma 8.4]{Mat} yields that there are $\gg\delta^{O(c)}K^d$ of integers $n\leq2^{3d}K^d$ such that $\set{q^dn\alpha\mrd{\Z}}$ stays in an interval of length $O(\delta^{-O(c)}\bigbrac{\frac{K}{N}}^{d})$. Then  making use of \cite[Lemma 1.1.14]{Tao} together with the fact $q\ll\delta^{-O(c)}$ ensures that there is $0<q_d\ll\delta^{-O(c)}$ such that $\norm{q_d\alpha_d}\ll\delta^{-O(c)}N^{-d}$.

We now assume that there is a group of integers $0<q_{j+1},\dots,q_d\ll\delta^{-O(c)}$ such that $\norm{q_i\alpha_i}\ll\delta^{-O(c)}N^{-i}$ holds for all $j<i\leq d$. And consider the case $j$. Also pigeonhole  gives some $p'\sim K$ such that for at least $\delta^{O(c)}$ proportion of $p\sim K$ such that $\bigset{(\sum_{j\leq i\leq d}q^jp^ja_p^{i-j}\alpha_i)\mrd{\Z}}$ lies in an interval of length at most $O\bigbrac{\delta^{-O(c)}\bigbrac{\frac{K}{N}}^{j}}$. Whilst,  from the assumption one has when $i>j$,
\[
\norm{q^jp^ja_p^{i-j}\alpha_i}\leq p^ja_p^{i-j}\norm{q^j\alpha_i}\ll\delta^{-O(c)}\Bigbrac{\frac{K}{N}}^i,
\]
 as $|a_p|\leq qK$, $p\sim K$ and $\norm{q_i\alpha_i}\ll\delta^{-O(c)}N^{-i}$ for some $q_i\ll\delta^{-O(c)}$. This means that the sequence $\set{q^jp^ja_p^{i-j}\alpha_i\mrd\Z}$ moves very slowly apart from the interval $\bigset{(\sum_{j\leq i\leq d}q^jp^ja_p^{i-j}\alpha_i)\mrd{\Z}}$, and thus, we can conclude that there are at least $\delta^{O(c)}$ proportion of $p\sim K$ such that $\set{q^jp^j\alpha_j\mrd{\Z}}$ stays in an interval of length at most $O\bigbrac{\delta^{-O(c)}\bigbrac{\frac{K}{N}}^{j}}$. Another application of \cite[Lemma 8.4]{Mat} and \cite[Lemma 1.1.14]{Tao} gives an integer $0<q_j\ll\delta^{-O(c)}$ such that $\norm{q_j\alpha_j}\ll\delta^{-O(c)}N^{-j}$. Therefore, the claim follows from taking $q'$ as the least common multiple of $q_1,\dots,q_d$.

Now write $\psi'=q'\cdot \psi$, then we see that $\norm{\psi'\circ g(n)}\ll\delta^{-O(c)}\frac{n}{N}$ for every positive integer $n$. Let $N'=\delta^{O(cC)}N$ for some large $C\geq1$ such that when $n\in[N']$, we can get
\[
\norm{\psi'\circ g(n)}\leq\frac{1}{10}.
\]
We then set $F':H/\Lambda\to\C$ to be the function $F'=\eta\circ \psi'$, where $\eta:\T\to\C$ is a function with mean zero, bounded Lipschitz norm and equals 1 on the interval $[-\frac{1}{10},\frac{1}{10}]$. Then we have $\int_{H/\Lambda}F'=0$, $\norm{F'}_\lip\ll\delta^{-O(c)}$ and,
\[
\bigabs{\E_{n\in[N]}F'\bigbrac{g(n)\Lambda}}\geq\delta\norm{F'}_\lip,
\]
provided that $c$ is sufficiently small. Thereby, we conclude the proof by noting that above inequality contradicts with the assumption that $\bigbrac{g(n)\Lambda}_{n\in[N]}$ is $\delta$-equidistributed.

\end{proof}

\end{lemma}

We now back our attention to the representation (\ref{pn}) and utilizing the full strength of Lemma \ref{prose}.
Recalling that the sequence $\bigbrac{g_i(n)}_{n\in[N]}$ is totally $M^{-A}$-equidistributed for each $i$, where
\begin{align}\label{m}
\log N\leq M\ll(\log N)^{O_A(1)},	
\end{align}
and $q\leq M$ and $|I|\geq\frac{N}{qM^2}$, one can invoke Lemma \ref{prose} with $\delta=M^{-A}$. To begin with, we split the sum range over $p,p'\sim N_k$ into two cases according to whether $\#\bigset{m:m\in \frac{I}{\max\set{p,p'}}}\gg M^{-O(cA)}N$ or not. In the case of $\bigset{m:m\in \frac{I}{\max\set{p,p'}}}$ does not contain many elements (less than $CM^{-O(cA)}N$), one can show that the total contribution in $(\ref{total})$ is negligible.  In practice, The fact that $F_{i,j}$ is a bounded function leads to
\[
 \sum_{m\in \frac{I}{\max\set{p,p'}}}F_{i,j}\bigbrac{g_i(qpm+a_p)\Lambda_i}\bar{F_{i,j}\bigbrac{g_i(qp'm+a_{p'})\Lambda_i}}\ll\#\bigset{m:m\in \frac{I}{\max\set{p,p'}}}.
\]
From the perspective of the assumption, it is bounded by $O(M^{-O(cA)}N)$.
Substituting it into (\ref{pp'}) to obtain that (\ref{pp'}) is no more than
\begin{multline*}
\Bigbrac{M^{-O(cA)}N\sum_{p,p'\sim N_k}\log p\log p'|f(p)||f(p')|}^{1/2}=M^{-O(cA)}N^{1/2}\sum_{p\sim N_k}\log p\,|f(p)|\\\ll M^{-O(cA)}(NN_k)^{1/2},	
\end{multline*}
\;where the last inequality follows from condition (\ref{lp2}) and Cauchy-Schwarz inequality. Substituting this inequality and (\ref{f2}) into (\ref{pn}) yields that $(\ref{pn})$ is bounded by $O( M^{-O(cA)-1}W^{1/2}N)$, due to $X\leq\frac{N}{qM^2}$. Thus, it can be seen that the contribution to (\ref{total}) is bounded by
\[
qW^{1/2}M^{-O(cA)}\sum_{k\leq \log N}\#\bigset{(a_1,a_2)\in[0,W-1)^2:a_1a_2\equiv b\mrd{W}}\ll1
\]
just by using the facts that $W\leq(\log N)^C$, $q\leq M$ and (\ref{m}) and by enlarging $A>1$ if necessary.

Hence, we just need to consider those pairs $(p,p')$ for which $\Bigabs{\frac{I}{\max\set{p,p'}}}\gg M^{-O(cA)}N$. Decomposing the summation over prime pairs $(p,p')$ in (\ref{pp'}) according to whether $(p,p')$ in the exceptional set $\mathcal E_k$, it allows us to conclude that (\ref{pp'}) is indeed a half power of the following expression
\begin{multline*}
\Bigset{\sum_{(p,p')\in\mathcal E_{N_k}}+\sum_{(p,p')\in(N_k,2N_k]^2\backslash\mathcal E_{N_k}}}\log p\log p'|f(p)||f(p')|\\
\times\biggabs{ \sum_{m\in \frac{I}{\max\set{p,p'}}}F_{i,j}
\bigbrac{g_i(qpm+a_p)\Lambda_i}\bar{F_{i,j}\bigbrac{g_i(qp'm+a_{p'})\Lambda_i}}}.
\end{multline*}
The application of Lemma \ref{prose} shows that $\#\mathcal E_{N_k}\ll M^{-O(cA)} N_k^2$, and thus the above first summation term is bounded by
\[
\sum_{(p,p')\in(N_k,2N_k]^2}\log p\log p'|f(p)||f(p')|1_{(p,p')\in\mathcal E_{N_k}}\cdot\frac{|I|}{\max\set{p,p'}},
\]
due to $F_{i,j}$ is a bounded function. By making use of Cauchy-Schwarz inequality and condition (\ref{lp2}), as well as noting that  $|I|=\Theta\bigbrac{\frac{N}{qM^2}}$, one may find it is bounded by
\begin{multline*}
\frac{N}{qM^2N_k}\Bigbrac{\sum_{p\sim N_k}(\log p\,f(p))^2}\Bigbrac{\#\mathcal E_{N_k}}^{1/2}\\\ll  q^{-1}M^{-O(cA)}N\sum_{p\sim N_k}	\log p\,|f(p)|^2\ll M^{-O(cA)}q^{-1}NN_k.
\end{multline*}
Substituting the above inequality and (\ref{f2}) into representation (\ref{total}) and taking $A\geq1$ large enough, the total contribution to (\ref{total}) can be bounded by $O(1)$.

It, therefore, remains to calculate the case $\Bigabs{\frac{I}{\max\set{p,p'}}}\gg M^{-O(cA)}N$ and $(p,p')\not\in\mathcal E_{N_k}$. Another application of Lemma \ref{prose} gives that
\begin{multline*}
\sum_{(p,p')\not\in\mathcal E_{N_k}}\log p\log p'|f(p)||f(p')|\biggabs{\sum_{m\in \frac{I}{\max\set{p,p'}}}F_{i,j}\bigbrac{g_i(qpm+a_p)\Lambda_i}\bar{F_{i,j}\bigbrac{g_i(qp'm+a_{p'})\Lambda_i}}}\\
\ll M^{-O(cA)}\frac{N}{N_k}\biggbrac{\sum_{p\sim N_k}\log p\,f(p)}^2\ll M^{-O(cA)}\frac{N}{N_k}\sum_{p\sim N_k}\log p\sum_{p\sim N_k}\log p|f(p)|^2\\
	\ll M^{-O(cA)} NN_k.
\end{multline*}
And in this case it can also  be shown that the contribution  to (\ref{total}) is  bounded by $O(1)$.

To sum up, in order to prove Lemma \ref{pissmall} we split the prime pairs $(p,p')\in(N_k,2N_k]^2$ in formula (\ref{pp'}) into three cases: the interval $\frac{I}{\max\set{p,p'}}$ is short; there are many elements in the interval $\frac{I}{\max\set{p,p'}}$ but $(p,p')$ belongs to the exceptional set $\mathcal E_{N_k}$; and there are many elements in the interval $\frac{I}{\max\set{p,p'}}$ and $(p,p')$ is out of $\mathcal E_{N_k}$ so that (\ref{fij}) has $M^{-O(cA)}$ decay. We then complete the deduction by showing that in each case the contribution to (\ref{total}) is bounded by $O(1)$.

\end{proof}

\subsection{When $p$ is large}

The task of this subsection is to prove the following lemma. And, clearly, Theorem \ref{main} follows from  this lemma and Lemma \ref{pissmall}.

\begin{lemma}\label{pislarge}
With all parameters  the same as (\ref{decom}), we have
\[
\frac{\phi(W)}{WN}\sum_{U<p\leq N}\sum_{i,j}\sum_{pn\in W\cdot P_{i,j}+b}\log p\,f(p)f(n)F_{i,j}\bigbrac{g_i(\frac{pn-b}{W})\Lambda_i}\ll1.
\]	
\end{lemma}
\begin{proof}
The proof, to some extent, is similar to Lemma \ref{pissmall}. We also start with splitting the interval $U<p\leq N$ into dyadic ranges, say $p\sim U_k$ with $U_k=2^kU$ and $0\leq k\leq\frac{\log N}{3\log 2}$. For a fixed number $0\leq k\leq\frac{\log N}{3\log 2}$ and a fixed pair $(i,j)$, consider the following representation
\[
\twosum{p\sim U_k}{p\equiv a_1\mrd{Wq}}\log p\,f(p)\twosum{\frac{Wa_{i,j}+b}{p}\leq n\leq\frac{WqX+Wa_{i,j}+b}{p}}{n\equiv a_2\mrd{Wq}}f(n)F_{i,j}\Bigbrac{g_i\bigbrac{\frac{pn-b}{W}}\Lambda_i}.
\]
Here, as before, letting $P_{i,j}=a_{i,j}+q\cdot[X]$, and both $a_1$ and $a_2$ run over the residue classes modulus $Wq$ and satisfy the relationship $a_1a_2\equiv Wa_{i,j}+b\mrd{Wq}$. The application of Cauchy-Schwarz inequality shows that it is bounded by
\begin{multline}\label{pnn'}
\biggbrac{\twosum{p\sim U_k}{p\equiv a_1\mrd{Wq}}\log p\,|f(p)|^2}^{1/2}\biggbrac{\twosum{\frac{Wa_{i,j}+b}{U_k}\leq n,n'\leq\frac{WqX+Wa_{i,j}+b}{U_k}}{n\equiv n'\equiv a_2\mrd{Wq}}f(n)f(n')\times
\\\times\twosum{p\sim U_k}{p\equiv a_1\mrd{Wq}}\log p\,F_{i,j}\Bigbrac{g_i\bigbrac{\frac{pn-b}{W}}\Lambda_i}\bar{F_{i,j}\Bigbrac{g_i\bigbrac{\frac{pn'-b}{W}}\Lambda_i}}}^{1/2}.	
\end{multline}
Condition (\ref{lp2}) implies that the first factor is bounded by $O(U_k^{1/2})$. After changing of variables, we see that the sum over $p$ in above second factor is indeed bounded by
\[
\sum_{m\sim\frac{U_k}{Wq}}\Lambda(Wqm+a_1)F_{i,j}\bigbrac{g_i(qmn+c_n)\Lambda_i}\bar{F_{i,j}\bigbrac{g_i(qmn'+c_{n'})\Lambda_i}},
\]
since the contribution of the higher powers is $O(U_k^{1/2+\eps})$  which is admissible in view of the estimate of the above expression, see (\ref{bound}).
When $(a_1,Wq)\neq1$, since $F_{i,j}$ is a bounded function, it is easy to see that in this case above expression is bounded by $\log U_k\cdot\#\set{p\sim U_k:p\equiv a_1\mrd{Wq}}\ll\log U_k$, thus the contribution is negligible.

 In the following, we can assume that $(a_1,Wq)=1$. A possible way to deal with this expression is to follow the idea of \cite[Lemma 9.5]{Mat}, to transfer $|\E_{m}\Lambda(Wqm+a_1)F(g'(m)\Gamma)|$ to $|\E_{m}F(g'(m)\Gamma)|$ with some  error term. In \cite[Lemma 9.5]{Mat} this error term is double logarithmic decay which is far too large than we expect. Indeed, one can check that to make Theorem \ref{main} has logarithmic decay, this error term should be strongly logarithmic decay. This seems achievable, because \cite[theorem 2.7]{TT} shows that the average of $\Lambda-\Lambda_{\mathrm{Siegel}}$ and polynomial nilsequences has pseudopolynomial decay. It, thus,  needs to calculate the average $\E_{n\in[N]}(\Lambda_{\mathrm{Siegel}}-1)(Wn+b)F(g'(n)\Gamma)$. \cite[Proposition 7.1]{TT} tells us that it is bounded by $O\bigbrac{M^{O(1)}q_{\mathrm{Siegel}}^{-O(1)}}$. Because in our case $M\ll(\log N)^{O(1)}$,  if $q_{\mathrm{Siegel}}\ll_A(\log N)^{-A}$, we have the chance to show that $\E_{n\in[N]}(\Lambda_{\mathrm{Siegel}}-1)(Wn+b)F(g'(n)\Gamma)$ is logarithmic decay. However, the assumption $q_{\mathrm{Siegel}}\ll_A(\log N)^{-A}$ will never happen in the case of  $W\ll(\log N)^{O(1)}$, on noting from the analysis surrounds formula (5.6) in \cite[P. 19]{TT}.\footnote{We would like to thank Joni Ter\"av\"ainen for explaining their paper to us.}

Our strategy is as follows. When $p$ is a large number, that is $p>U=N^{2/3}$, due to $pn,pn'\leq WN$, both $n$ and $n'$ should be small numbers. Inspired by the proof idea of Lemma \ref{pissmall}, when polynomial $g$ is highly equidistributed, we guess for almost all pairs of $(n,n')$ the product polynomial $\bigbrac{g(n\cdot),g(n'\cdot)}$ should be equidistributed. It, therefore, can be reduced to consider the correlation of an equidistributed nilsequence with the shifted von Mangoldt function. And this can be calculated by the so-called bi-linear form method adapted to polynomial nilsequences, see \cite[Section 3]{GT12b} as an example.

For technical reason, we would decompose the summation interval of $n,n'$ in the following representation into dyadic ranges, say $n,n'\sim V_l$ with $V_l=2^l\frac{Wa_{i,j}}{U_k}$ and $0\leq l\ll\log N$.
\begin{align}\label{splitnn'}
\sum_n f(n)f(n')\sum_{m\sim\frac{U_k}{Wq}}\Lambda(Wqm+a_1)F_{i,j}\bigbrac{g_i(qmn+c_n)\Lambda_i}\bar{F_{i,j}\bigbrac{g_i(qmn'+c_{n'})\Lambda_i}},
\end{align}
where both $n$ and $n'$ are supported in the interval $[\frac{Wa_{i,j}+b}{U_k},\frac{WqX+Wa_{i,j}+b}{U_k}]$ and subject to $n\equiv n'\equiv a_2\mrd{Wq}$.
Let $g_{i;n,n'}:\Z\to H_i/\Lambda_i\times H_i/\Lambda_i$ be the polynomial defined by
\[
g_{i;n,n'}(m)=\bigbrac{g_i(qnm+c_n),g_i(qn'm+c_{n'})}.
\]
\subsection*{Claim} There is a small constant $0<c<1$ such that the following statement holds uniformly for $V_l$ with $0\leq l\ll\log N$. Suppose that $\bigbrac{g(m)}_{m\sim\frac{U_k}{Wq}}$ is totally $M^{-A}$-equidistributed. When the pairs $(n,n')\in(V_l,2V_l]$ is outside of a subset of size bounded by $O(M^{-O(cA)}V_l^2)$, the corresponding product polynomial sequence $\bigbrac{g_{i;n,n'}(m)}_{m\sim\frac{U_k}{Wq}}$ is totally $M^{-O(cA)}$-equidistributed.

Since the treatment for every $i$ is the same, in the following we fix one such $i$ and omit the letter $i$. Similar to the proof of Lemma \ref{prose}, we assume that the sequence $\bigbrac{g_{n,n'}(m)}_{m\sim\frac{U_k}{Wq}}$ fails to be $M^{-O(cA)}$-equidistributed. Then from quantitative Leibman theorem (Theorem \ref{leibman}), there is a non-trivial horizontal character of modulus $0<|\psi|\ll M^{-O(cA)}$ such that
\[
\norm{\psi\circ g_{n,n'}}_{C^\infty[\frac{U_k}{Wq},\frac{2U_k}{Wq}]}\ll M^{-O(cA)}.
\]
One can see from the proof of Lemma \ref{prose} that this contradicts with the assumption $\bigbrac{g(m)}_{m\sim\frac{U_k}{Wq}}$ is $M^{-A}$-equidistributed when we take $0<c<1$ sufficiently small. Just need to pay attention to one thing, the prime pairs $(p,p')$ in Lemma \ref{prose} are replaced by integer pairs $(n,n')$ here, but this doesn't matter because it can be achieved by replacing \cite[Lemma 8.4]{Mat} by \cite[Lemma 3.3]{GT12b} in the proof. The claim then follows by an application of \cite[Lemma 7.2]{Mat}, which asserts that if $(g(m))_{m\in[N]}$ is $\delta$-equidistributed, then there is a number $0<c<1$ such that $(g(m))_{m\in[N]}$ is totally $\delta^c$-equidistributed.

Now we write $\widetilde F_{i,j}(\gamma,\gamma')=F_{i,j}(\gamma)\bar{F_{i,j}(\gamma')}$ as a Lipschitz function supported on $H_i/\Lambda_i\times H_i/\Lambda_i$. Then we have the  information that $\int_{H_i/\Lambda_i\times H_i/\Lambda_i}\widetilde F_{i,j}=0$ and $\|\widetilde F\|_\lip\ll M^{-O(1)}$. Next, we aim   is to prove  when $\bigbrac{g_{i;n,n'}(m)}_{m\sim\frac{U_k}{Wq}}$ is totally $M^{-O(cA)}$-equidistributed, the following inequality holds
\begin{align}\label{lastone}
\sum_{m\sim\frac{U_k}{Wq}}\Lambda(Wqm+a_1)\widetilde F_{i,j}(g_{i;n,n'}(m)\Lambda_i\times\Lambda_i)\ll M^{-O(cA)}\frac{U_k}{Wq}.	
\end{align}
Here, $a_{1}$ is less than $Wq$ and coprime with $Wq$. This, indeed, follows from the next lemma. And a brief proof of the below lemma will be given after the proof of Lemma \ref{pislarge}.

\begin{lemma}[von Mangoldt function disjoints with equidistributed nilsequences]\label{Mangoldt}
Let $0<\delta<1/2$ be a parameter. Suppose that $G/\Gamma$ is a nilmanifold of finite dimension $m_G$, $G_\bullet$ is a filtration of $G/\Gamma$ and of degree $d$, and $\mathcal Y$ is a $\delta^{-O(1)}$-rational Mal'cev basis adapted to $G_\bullet$. Suppose that $g\in\poly(\Z,G/\Gamma)$ is a polynomial and $(g(n))_{n\in[N]}$ is totally $\delta$-equidistributed. Then for any coprime pair $1\leq b\leq W\ll\delta^{-O(1)}$ and any Lipschitz function $F:\Z\to\C$ with $\int_{G/\Gamma}F=0$ and $\norm{F}_\lip\ll\delta^{-O(1)}$, we have
\[
\E_{n\in[N]}\Lambda(Wn+b)F(g(n)\Gamma)\ll\delta^{-O(1)}.
\]

\end{lemma}

We now splice all the ingredients to finish the proof of Lemma \ref{pislarge}. For a fixed number $0\leq l\ll\log N$, divide the pairs $(n,n')\in(V_l,2V_l]^2$ into two sets. One is that (\ref{lastone}) holds for each pair $(n,n')$ in this set, and we call the other one as the exceptional set $\mathcal E_l$. In view of the claim, the size of the exceptional set $\mathcal E_l$ is at most $O(M^{-O(cA)}V_l^2)$, besides, for $(n,n')\in\mathcal E_l$, we notice that
\begin{align}\label{bound}
\sum_{m\sim\frac{U_k}{Wq}}\Lambda(Wqm+a_1)\widetilde F_{i,j}(g_{i;n,n'}(m)\Lambda_i\times\Lambda_i)\ll\sum_{m\sim\frac{U_k}{Wq}}\Lambda(Wqm+a_1)\ll\frac{U_k}{\phi(Wq)},
\end{align}
as $\widetilde F_{i,j}$ is bounded and $(Wq,a_1)=1$. One can now see from (\ref{lastone}), together with above information for exceptional set, that (\ref{splitnn'}) is bounded by
\[
\sum_{0\leq l\ll\log N}\Bigset{M^{-O(cA)}\frac{U_k}{Wq}\sum_{n,n'\sim V_l}|f(n)||f(n')|+\frac{U_k}{\phi(Wq)}\sum_{(n,n')\in(V_l,2V_l]^2}|f(n)||f(n')|1_{\mathcal E_l}(n,n')}.
\]
We now compute above two sum terms in turns.  Summing over the dyadic intervals $(V_l,2V_l]$, and then using Cauchy-Schwarz inequality and condition (\ref{fl2}),  one has the first term is bounded by
\begin{align*}
&\qquad M^{-O(cA)}\frac{U_k}{Wq}\biggbrac{\sum_{\frac{Wa_{i,j}+b}{U_k}\leq n\leq\frac{WqX+Wa_{i,j}+b}{U_k}}|f(n)|}^2\\
&\ll M^{-O(cA)}\frac{U_k}{Wq}\biggbrac{\sum_{n\leq\frac{WqX+Wa_{i,j}+b}{U_k}}|f(n)|-\sum_{n\leq\frac{Wa_{i,j}+b}{U_k}}|f(n)|}^2\ll M^{-O(cA)}\frac{WqX^2}{U_k}.
\end{align*}
Besides, Cauchy-Schwarz inequality and condition (\ref{fl2}) and summing over the dyadic intervals $(V_l,2V_l]$ also shows that the second term is bounded by
\begin{multline*}
\frac{U_k}{\phi(Wq)}\Bigbrac{\sum_{l\ll\log N}\sum_{n\in V_l}|f(n)|^2}\Bigbrac{\sum_{l\ll\log N}\sum_{(n,n')\in(V_l,2V_l]^2}1_{\mathcal E_l}(n,n')^2}^{1/2}\\
\ll	\frac{U_k}{\phi(Wq)}\biggbrac{\sum_{n\leq\frac{WqX+Wa_{i,j}+b}{U_k}}|f(n)|^2-\sum_{n\leq\frac{Wa_{i,j}+b}{U_k}}|f(n)|^2}(\log N)^{1/2}|\mathcal E_l|^{1/2}\\\ll M^{-O(cA)}(\log N)^{1/2}\frac{NX}{U_k}.
\end{multline*}

Recalling that $X\gg\frac{N}{qM^2}$ and $\log N\leq M\ll(\log N)^{O_A(1)}$, when we take $A>1$ sufficiently large, the factors $(\log N)^{1/2}$ and $N/X=qM^2$ can be absorbed by $M^{-O(cA)}$, thus, we can assume that both of the above two terms are bounded by $O\bigbrac{M^{-O(cA)}\frac{X^2}{U_k}}$. Substitute it into (\ref{pnn'}), as well as recalling that the first factor in (\ref{pnn'}) is no more than $O(U_k^{1/2})$, we can conclude that (\ref{pnn'}) is bounded by $O(M^{-O(cA)}X)$.  Splicing the intervals $(U_k,2U_k]$ with $0\leq k\leq\frac{\log N}{3\log2}$, one therefore has
\begin{multline*}
\frac{\phi(W)}{WN}\sum_{U<p\leq N}\sum_{i,j}\sum_{pn\in W\cdot P_{i,j}+b}\log p\,f(p)f(n)F_{i,j}\bigbrac{g_i(\frac{pn-b}{W})\Lambda_i}\\
\ll	\frac{\phi(W)}{WN}\sum_{k\ll\log N}\sum_{i,j}\twosum{a_1,a_2\mrd{Wq}}{a_1a_2\equiv Wa_{i,j}+b\mrd{Wq}}M^{-O(cA)}X\ll1,
\end{multline*}
as $X=\Theta(\frac{N}{qM^2})$ and letting $A>1$ sufficiently large. Therefore, we complete the proof of this lemma.

\end{proof}

\noindent\emph{A brief proof of Lemma \ref{Mangoldt}.}

On recalling the well-known Vaughan's identity \cite{Vau}
\[
\Lambda(n)=\Lambda(n) 1_{n \leq N^{1 / 3}}-\sum_{d \leq N^{2 / 3}} a_{d} 1_{d \mid n}+\sum_{d \leq N^{1 / 3}} \mu(d) 1_{d \mid n} \log \frac{n}{d}+\sum_{d, w>N^{1 / 3}} \Lambda(d) b_{w} 1_{d w=n},
\]
where $a_d=\sum_{bc=d:b,c\leq N^{1/3}}\mu(b)\Lambda(c)$ and $b_w=\sum_{c|w:c>N^{1/3}}\mu(c)$. The first term is a negligible one, we call the second and fourth term as Type I sum and  Type II sum respectively, besides, the third one can be expressed as a convex combination of Type I sum by using the identity
\[
\log \frac{n}{d}=\log N-\int_{1}^{N} 1_{t>n} \frac{d t}{t}-\log d,
\]
and then by absorbing all the various logarithmic factors into the divisor-bounded coefficients.

Therefore, the expression $\sum_{n\in[N]}\Lambda(Wn+b)F(g(n)\Gamma)$ can be expressed as follows,
\[
\twosum{n\in P}{n\leq N^{1/3}}\Lambda(n)F\Bigbrac{g\bigbrac{\frac{n-b}{W}}\Gamma}+\sum_{n\in P}\twosum{d\leq N^{2/3}}{d|n}a_dF\Bigbrac{g\bigbrac{\frac{n-b}{W}}\Gamma}+\twosum{dw\in P}{d,w\geq N^{1/3}}a_db_wF\Bigbrac{g\bigbrac{\frac{dw-b}{W}}\Gamma},
\]
where $P=b+W\cdot[N]$ is the arithmetic progression,  $a_d\ll (\log N)^{O(1)}\tau^{O(1)}(d)$, and $b_w\ll (\log N)^{O(1)}\tau^{O(1)}(w)$. Plainly, the first term is bounded by $O(N^{1/2})$ (say) so is negligible, and the latter two terms can be dealt with in the same manner of \cite[section 3]{GT12b}, also of \cite[Section 7]{TT}.

\vspace{2mm}

\qed

\section{Proof of Theorem \ref{lfunction} and Theorem \ref{mulfunction}}

We first give some standard facts about $L$-functions which can be found in \cite[Section 2]{RS} and \cite[Chapter 5]{IK}.

Let $m\geq 2$ be an integer, $\pi=\otimes_p \pi_p$ be a normalized irreducible cuspidal automorphic representations of $GL_m$ over $\Q$, which means that $\pi$ has unitary central character.
The $L$-function $L(s,\pi)$ associated to $\pi$ is defined by
\[L(s,\pi)=\sum_{n=1}^{\infty}\frac{\lambda_\pi(n)}{n^s}=\prod_p\prod_{j=1}^{m}\left(1-\frac{\alpha_{j,\pi}(p)}{p^s}\right)^{-1},\]
for some suitable complex numbers $\alpha_{j,\pi}(p)$. The Dirichlet series converge absolutely for $\Re(s)>1$.

There is also an archimedean local factor $L(s, \pi_\infty)$. There are $m$ complex Langlands parameters $\mu_\pi(j)$, we define
\[L(s, \pi_\infty)=\pi^{-\frac{m s}{2}} \prod_{j=1}^{m} \Gamma\left(\frac{s+\mu_{\pi}(j)}{2}\right).\]
The generalized Ramanujan conjecture and Selberg conjecture assert that
\[|\alpha_{j,\pi}(p)|\leq 1 \ \textrm{and} \ \Re(\mu_{\pi}(j))\geq 0.\]
The Ramanujan conjecture was proved by Deligne \cite{Del} for holomorphic cusp form on $GL_2$. According to Luo, Rudnick and Sarnak \cite{LRS} and M\"uller and
Speh \cite{MS}, we know that
\begin{equation}\label{bound-alpha}
|\alpha_{j,\pi}(p)|\leq p^{\frac{1}{2}-\frac{1}{m^2+1}} \ \textrm{and} \ \Re(\mu_{\pi}(j))\geq -(\frac{1}{2}-\frac{1}{m^2+1}).
\end{equation}

In order to define the functional equation for $L(s,\pi)$, we give the contragredient of $\pi$. Let $\widetilde{\pi}$ be the contragredient of $\pi$ which is also an irreducible cuspidal automorphic representations of $GL_m$ over $\Q$. For each $p\leq \infty$, we have
\[\big\{\alpha_{j, \pi}(p): 1 \leq j \leq m\big\}=\big\{\overline{\alpha_{j, \pi}(p)}: 1 \leq j \leq m\big\}\]
and
\[\big\{\mu_{\widetilde{\pi}}(j): 1 \leq j \leq m\big\}=\big\{\overline{\mu_{\pi}(j)}: 1 \leq j \leq m\big\}.\]
Define the completed $L$-function
\[\Lambda(s, \pi)=N_{\pi}^{s / 2} L(s, \pi) L\left(s, \pi_{\infty}\right),\]
where $N_\pi$ be the conductor of $\pi$. Then, $\Lambda(s, \pi)$ extends to an entire function and satisfies the functional equation
\[\Lambda(s, \pi)=\xi (\pi) \Lambda(1-s, \tilde{\pi}),\]
where $\xi(\pi)$ is a complex number of modulus 1.

Let $\chi$ be a primitive Dirichlet character modulo $q$. Then the twisted $L$-function is defined by
\[L(s, \pi \times \chi)=\prod_{p} \prod_{j=1}^{m}\left(1-\frac{\alpha_{j, \pi \times \chi}(p)}{p^{s}}\right)^{-1}.\]
When $p\nmid q$, we have
\[\left\{\alpha_{j, \pi \times \chi}(p): 1 \leq j \leq m\right\}=\left\{\alpha_{j, \pi}(p) \chi(p): 1 \leq j \leq m\right\}.\]
Thus we get
\begin{equation}\label{Lspichi}
	\begin{aligned}
		\sum_{n=1}^{\infty} \frac{\lambda_{\pi}(n) \chi(n)}{n^{s}} &=\prod_{p} \prod_{j=1}^{m}\left(1-\frac{\alpha_{j, \pi}(p) \chi(p)}{p^{s}}\right)^{-1} \\
		&=L(s, \pi \times \chi) \prod_{p \mid q} \prod_{j=1}^{m}\left(1-\frac{\alpha_{j, \pi \times \chi }(p)}{p^{s}}\right).
	\end{aligned}
\end{equation}
We also need the following convexity bound for $L(s,\pi\times\chi)$, see \cite[Lemma 3.2]{JLW},
\begin{equation}\label{convbound-l}
	L(\frac{1}{2}+it,\pi\times\chi)\ll (q(1+|t|))^{\frac{m}{4}+\varepsilon}.
\end{equation}

Next, we prove that $\lambda_\pi (n)\in \mathcal M^\prime$ with $W=1$, then Theorem \ref{lfunction} directly follows from Theorem \ref{main}. The conditions (\ref{fl2}) and (\ref{lp2}) can be easily be deduced
from the Rankin-Selberg theory and the prime number theorem of Rankin-Selberg
L-functions, see \cite[(5.2) and Page 631]{JLW}, i.e.,
\begin{equation}\label{lampi-l2}
	\sum_{n \leq N}\left|\lambda_{\pi}(n)\right|^{2} \ll N,
\end{equation}
and
\begin{equation}\label{bound-lam-p}
	\sum_{p \leq N}\left|\lambda_{\pi}(p)\right|^{2} \log p \ll  N.
\end{equation}

Then, we just need to prove that $\lambda_\pi(n)$ satisfy condition (\ref{w-equi}) with $W=1$.
Take $P=\{n\equiv b (\bmod q), n\leq N\}$ with $q\leq (\log N)^C$, we consider the following sum
\[	\sum_{n\leq N \atop n\equiv b(\bmod q)}\lambda_{\pi}(n).\]
We will consider the following two cases according to $(b,q)=1$ or $(b,q)>1$.

Case I: $(b,q)=1.$ Applying the orthogonality of Dirichlet characters, we have
\begin{equation}\label{lam-arithprog}\sum_{n\leq N \atop n\equiv b(\bmod q)}\lambda_{\pi}(n)=\frac{1}{\varphi(q)}\sum_{\chi\bmod q}\bar{\chi}(b)\sum_{n\leq N }\lambda_{\pi}(n)\chi(n).\end{equation}
It suffices to estimate the sum
\begin{equation}\label{pichi}
	\sum_{n \leq N} \lambda_{\pi}(n) \chi(n)
\end{equation}
for any $\chi(\bmod q)$. This is the same with \cite[(5.4)]{JLW}, they proved that

\begin{equation}\label{bound-lamchi}
	\sum_{n\leq N}\lambda_{\pi}(n)\chi(n)\ll q^{\frac{m}{2m+4}+\varepsilon}N^{\frac{m+1}{m+2}+\varepsilon}.
\end{equation}
Taking (\ref{bound-lamchi}) into (\ref{lam-arithprog}), we get
\begin{equation}\label{bound-lambdarith}
\sum_{n\leq N \atop n\equiv b(\bmod q)}\lambda_{\pi}(n)\ll q^{\frac{m}{2m+4}+\varepsilon}N^{\frac{m+1}{m+2}+\varepsilon}.
\end{equation}
This is also true for $q=1.$ We confirm that $\lambda_{\pi}(n)$ satisfy condition (\ref{w-equi}) with $W=1$.

Case II: $(b,q)>1$. First, we consider that $n$ runs over square-free numbers. Let $d=(b,q), \ b=b^\prime d$ and $q=q^\prime d$, then $(b^\prime,q^\prime)=1$ and $(d,n/d)=1$.  Denote $\chi_d$  be the principal character modulo $d$. Then we have
\begin{equation}\label{equ-squarefree}
\begin{aligned}
	\threesum{n \leq N}{n \equiv b(\bmod q)}{n \ \textrm{square-free}} \lambda_{\pi}(n)&=\lambda_{\pi}(d) \quad \threesum{l \leq N / d }{l\equiv b^\prime (\bmod q^\prime)} {l \ \textrm{square-free}}\quad \lambda_{\pi}(l)\chi_d(l)\\
	&=\frac{\lambda_{\pi}(d)}{\varphi(q^\prime)}\sum_{\chi\bmod q^\prime }\bar{\chi}(b^\prime)\twosum{l \leq N / d}{l \ \textrm{square-free}} \lambda_{\pi}(l)(\chi\chi_d)(l)
	\end{aligned}
\end{equation}
Since $\chi\chi_{d}$ is a character modulo $q^\prime d=q$, we use (\ref{bound-lamchi}) to the inner sum, we have
\begin{equation}\label{bound-squarefree}
	\threesum{n\leq N}{ n\equiv b(\bmod q)}{n \ \textrm{square-free}}\lambda_\pi(n)\ll q^{\frac{m}{2m+4}+\varepsilon}N^{\frac{m+1}{m+2}+\varepsilon}.
\end{equation}
 Here we use $\lambda_{\pi}(d)\ll d^{\frac{1}{2}-\frac{1}{m^2+1}+\varepsilon}.$

 Next we remove the restriction that $n$ is square-free. Any arbitrary positive integer $n$ can be represented in a unique way as
 \[n=k l, \quad k \text { is square-free, } \quad l \text { is square-full, } \quad(k, l)=1.\]
 Hence we have
 \begin{equation*}
 	\begin{aligned}
 		\sum_{n \leq N \atop n \equiv b(\bmod q)} \lambda_{\pi}(n)&=\threesum{l k \leq N}{lk \equiv b(\bmod q)}{l \ \textrm{square-full}, \ k \ \textrm{square-free}} \lambda_{\pi}(l) \lambda_{\pi}(k)=\sum_{l \leq N \atop l \textrm{ square-full }} \lambda_{\pi}(l) \threesum{k \leq N / l }{m k \equiv b(\bmod q)}{k \ \textrm{square-free}} \lambda_{\pi}(k)\\
 		&=\sum_{l \leq N\atop 	l \ \textrm{square-full}} \lambda_{\pi}(l)  \threesum{k \leq N / l }{ k \equiv \frac{b}{(l, q)} \overline{\frac{l}{(l, q)}}\left(\bmod \frac{q}{(l, q)}\right)}{k \ \textrm{square-free}} \lambda_{\pi}(k).  	
 	\end{aligned}
 \end{equation*}
 We split the sum over $l$ into two parts: $1\leq l\leq N^{4/5}$ and $N^{4/5}< l\leq N$. If $1\leq l\leq N^{4/5}$, then $N/l\geq N^{1/5}$. We use (\ref{bound-squarefree}) for the inner sum, we have
 \begin{equation}\label{lsmall}
 \threesum{n \leq N}{n \equiv b(\bmod q)}{l\leq N^{4/5}, \ l \ \textrm{square-full}} \lambda_{\pi}(n)\ll q^{\frac{m}{2m+4}+\varepsilon}N^{\frac{m+1}{m+2}+\varepsilon}\sum_{l \leq N^{4/5}\atop 	l\  \textrm{square-full}} \frac{|\lambda_{\pi}(l)|}{l^{\frac{m+1}{m+2}+\varepsilon}}\ll q^{\frac{m}{2m+4}+\varepsilon}N^{\frac{5m+9}{5m+10}+\varepsilon}.
 \end{equation}
 Here we use (\ref{lampi-l2}) and the Cauchy-Schwarz inequality to get
 \[\sum_{l \leq N^{4/5}\atop 	l \ \textrm{ square-full}} \frac{|\lambda_{\pi}(l)|}{l^{\frac{m+1}{m+2}+\varepsilon}}\ll N^{\frac{4}{5(m+2)}}\left(\sum_{l\leq N^{4/5}}\frac{1}{l}\right)^\frac{1}{2}\left(\sum_{l\leq N^{4/5}}\frac{|\lambda_{\pi}(l)|^2}{l}\right)^\frac{1}{2}\ll N^{\frac{4}{5(m+2)}}\log N. \]

 If $N^{4/5}<l\leq N$, we use (\ref{lampi-l2}) and the Cauchy-Schwarz inequality
\begin{equation}\label{llarge}
	\begin{aligned}
	\threesum{n \leq N}{n \equiv b(\bmod q)}{N^{4/5}<l\leq N, \ l\ \textrm{square-full}} \lambda_{\pi}(n)&\ll \sum_{N^{4/5}<l \leq N \atop l \ \textrm{square-full}}|\lambda_{\pi}(l)|\twosum{k\leq N^{1/5}}{k\ \textrm{square-free}}|\lambda_{\pi}(k)|\\
	&\ll \left(\sum_{l \leq N \atop l \ \textrm{square-full}}1\right)^{1/2}\left(\sum_{l \leq N}|\lambda_{\pi}(l)|^2\right)^{1/2}\left(\sum_{k \leq N^{1/5} }1\right)^{1/2}\left(\sum_{k \leq N^{1/5} }|\lambda_{\pi}(k)|^2\right)^{1/2}\\
	&\ll N^{19/20}.
	\end{aligned}
\end{equation}
From (\ref{lsmall}) and (\ref{llarge}), we have
\[\sum_{n \leq N \atop n \equiv b(\bmod q)} \lambda_{\pi}(n)\ll q^{\frac{m}{2m+4}+\varepsilon}N^{\frac{5m+9}{5m+10}+\varepsilon}.\]
 We confirm that $\lambda_{\pi}(n)$ satisfy condition (\ref{w-equi}). This completes the proof of Theorem \ref{lfunction}.

 In order to prove Theorem \ref{mulfunction}, we prove that $\mu(n)\lambda_\pi (n)\in \mathcal M^\prime$ with $W=1$. From (\ref{lampi-l2}) and (\ref{bound-lam-p}), we know that
\begin{equation}
\sum_{n \leq N}\left|\mu(n)\lambda_{\pi}(n)\right|^{2}\leq	\sum_{n \leq N}\left|\lambda_{\pi}(n)\right|^{2} \ll N,
\end{equation}
and
\begin{equation}
	\sum_{p \leq N}\left|\mu(p)\lambda_{\pi}(p)\right|^{2} \log p \ll  N.
\end{equation}
This says that $\mu(n)\lambda_{\pi}(n)$ satifies conditions (\ref{fl2}) and (\ref{lp2}).

Next we consider the following sum
\[	\sum_{n\leq N \atop n\equiv b(\bmod q)}\mu(n)\lambda_{\pi}(n),\]
with $q\leq (\log N)^C$. We use the recent result of Jiang, L\"u and Wang \cite[Section 5]{JLW21}, for $\pi$ is self-dual and $\pi \ncong \pi\otimes\chi$ for any quadratic primitive character $\chi$, they prove that
\begin{equation}\label{bound-mupichi}
	\sum_{n \leqslant N} \mu(n) \lambda_{\pi}(n) \chi(n) \ll N \exp \left( -c \sqrt{\log N}\right),
\end{equation}
for any Dirichlet character $\chi(\bmod \ q)$, where $c>0$ is a constant.

 Let $d=(b,q), \ b=b^\prime d$ and $q=q^\prime d$, then $(b^\prime,q^\prime)=1$. Denote $\chi_d$ be the principal character modulo $d$. Then we have
 \begin{equation}
 	\begin{aligned}
 		\sum_{n\leq N\atop n\equiv b(\bmod q)} \mu(n)\lambda_{\pi}(n) &=\sum_{n\leq N/d \atop n \equiv b^\prime\left(\bmod q^\prime\right)} \mu(dn)\lambda_{\pi}\left(d n\right) \\
 		&= \mu(d)\lambda_{\pi}(d) \sum_{n\leq N/d\atop n\equiv b^\prime (\bmod q^\prime)} \mu(n)\lambda_{\pi}\left(n\right) \chi_{d}\left(n\right) \\
 		&=\frac{\mu(d)\lambda_{\pi}(d)}{\varphi\left(q^\prime\right)} \sum_{\chi\left(\bmod q^\prime\right)} \bar{\chi}\left(b^\prime\right) \sum_{n \leqslant N / d}\mu(n) \lambda_{\pi}\left(n\right) \left(\chi \chi_{d}\right)\left(n\right)
 	\end{aligned}
 \end{equation}
 Since $\chi\chi_d$ is a character modulo $q^\prime d=q$, we use (\ref{bound-mupichi}) to the inner sum, we have
 \[\sum_{n\leq N\atop n\equiv b(\bmod q)} \mu(n)\lambda_{\pi}(n)\ll N\exp \left( -c \sqrt{\log (N/q)}\right).\]
 Here we use $\lambda_{\pi}(d)\ll d^{\frac{1}{2}-\frac{1}{m^2+1}+\varepsilon}$. We confirm that $\mu(n)\lambda_{\pi}(n)$ satisfy condition (\ref{w-equi}) with $W=1$. This completes the proof of Theorem \ref{mulfunction}.

\bibliographystyle{plain} 

\renewcommand{\bibname}{} 

\bibliography{discorrelation}

\end{document}